\documentclass[10pt,a4paper]{amsart}
\usepackage{fourier}
\usepackage{ulem}
\usepackage{amsthm}
\usepackage{mathrsfs}
\usepackage{graphicx}
\usepackage{stmaryrd}
\usepackage{float}
\usepackage{amsmath}
\usepackage[all]{xy}\usepackage{xypic}
\usepackage{hyperref}
\hypersetup{
  pdftitle   = {},
  pdfauthor  = {},
  pdfcreator = {\LaTeX\ with package \flqq hyperref\frqq}
}
\DeclareMathAlphabet{\mathpzc}{OT1}{pzc}{m}{it}

\newcommand{\got}{\mathfrak}

\newtheorem{theorem}{Theorem}[section]
\newtheorem*{theorem*}{Theorem}
\newtheorem{proposition}[theorem]{Proposition}
\newtheorem{lemma}[theorem]{Lemma}
\newtheorem*{lemma*}{Lemma}
\newtheorem{corollary}[theorem]{Corollary}

\newtheorem*{conjecture*}{Conjecture}

\theoremstyle{definition}
\newtheorem{definition}[theorem]{Definition}

\theoremstyle{remark}
\newtheorem{remark}[theorem]{Remark}

\DeclareMathOperator{\spam}{span}

\DeclareMathOperator{\ad}{ad}
\DeclareMathOperator{\rk}{rk}
\DeclareMathOperator{\tr}{tr}

\DeclareMathOperator{\diag}{diag}

\numberwithin{equation}{section}

\begin{document}


\title[Invariant almost complex geometry on flag manifolds]{Invariant almost complex geometry on flag manifolds: geometric formality and Chern numbers}
\author{Lino Grama, Caio J.C. Negreiros and Ailton R. Oliveira}
\address{Department of Mathematics - IMECC, University of Campinas - Brazil}
\email{linograma@gmail.com, caione@ime.unicamp.br, ailton\_rol@yahoo.com.br }
\thanks{This research was supported by CNPq grant 476024/2012-9 and Fapesp grant no. 2012/18780-0. LG is also supported by Fapesp grant no. 2014/17337-0.}
\begin{abstract}
In the first part of this paper we study geometric formality for generalized flag manifolds, including full flag manifolds of exceptional Lie groups. In the second part we deal with the problem of the classification of invariant almost complex structures on generalized flag manifolds using topological methods.
\end{abstract}
\maketitle
\section{Introduction}
In this paper we study geometric formality and classification of almost complex structures on generalized flag manifolds (or K\"ahler C-spaces). This class of homogeneous spaces is defined taking the quotient $G/P$ of a complex simple non-compact Lie group $G$ by the normalizer of a parabolic sub-algebra $\mathfrak{p}$ of the Lie algebra $\mathfrak{g}=Lie(G)$. Equivalently, a generalized flag manifold is defined as $U/K$, where $U$ is the maximal compact sub-group of $G$ and $K=P\cap U$ is a centralizer of a torus. It is well known that generalized flag manifolds have a rich Riemannian and Hermitian geometry (see for instance \cite{besse}).

In the first part of this paper we study the problem of geometric formality for generalized flag manifolds. On a general Riemannian manifold, wedge products of harmonic forms are not usually harmonic. But there are some examples where this does happen, like compact globally symmetric spaces. Motivated by examples of closed surfaces of genus $\geq 2$ (in this case there are non-trivial harmonic 1-form for any metric, but every 1-form has zeros), Kotschick in \cite{Kots} introduced the notion of geometrically formal manifolds: a smooth manifold is geometrically formal if it admits a Riemannian metric for which all exterior products of harmonic forms are harmonic. 

Classical examples of geometrically formal manifolds are compact symmetric spaces and Stiefel manifolds (real, complex, quaternionic and octonionic). Geometric formality implies the formality in the sense of Sullivan, and in fact it is more restrictive. For instance, in \cite{Kote} Kotschick-Terzi\'c proved that all generalized symmetric spaces of compact simple Lie groups are formal in the sense of Sullivan, and that many of them are not geometrically formal.

Regarding non-geometric formality on flag manifolds the following examples are already known: full flag manifolds $G/T$ where $G$ is a classical Lie group ($SU(n), SO(n), Sp(n)$) or $G=G_2$ (\cite{Kote}); the family of generalized flag manifolds $SU(n+2)/S(U(n)\times U(1)\times U(1))$ (\cite{princ}); Wallach flag manifolds with positive sectional curvature (\cite{AZ}).

Our first result proves the non-geometric formality for full flag manifolds associated to exceptional Lie groups. Since these Lie groups complete the list of compact simple Lie groups, one can state the following result:

\begin{theorem*}
The full flag manifolds $G/T$, with $G$ compact simple Lie groups are not geometrically formal.
\end{theorem*}

We also provide a large family of examples of non-geometrically formal flag manifolds (for details, see Section \ref{sec-nonf}).

In the second part of the paper, we study Chern numbers of invariant almost complex structures on generalized flag manifolds. The classification of Hermitian structures on full flag manifolds was carried out by San Martin-Negreiros in \cite{SM}. In \cite{rita} San Martin-Silva study the invariant Nearly-K\"ahler structures on flag manifolds. We remark that in both works cited above, the Lie theoretical methods was used in a crucial way. On other hand, the classification of invariant  Hermitian structures on generalized flag manifolds remain an open problem. Therefore, it is a natural question to classify the invariant almost complex structures (or more generally Hermitian structures) on generalized flag manifolds. 

In this work we use characteristic classes in order to classify these invariant almost complex structures in some flag manifolds. More precisely, using the Chern numbers joint with a classical result due to Borel-Hirzebruch we obtain in some cases the classification of such invariant almost complex structures (up to conjugation and equivalence).
\begin{theorem*}
The following generalized flag manifolds 
$$SU(6)/S(U(1)\times U(2)\times U(3)), $$
$$SU(7)/S(U(1)\times U(2)\times U(4)), $$
$$SU(8)/S(U(1)\times U(2)\times U(5)), $$
$$SU(8)/S(U(1)\times U(3)\times U(4)) $$
have precisely $4$ invariant almost complex structures up to conjugation and equivalence; $3$ of them are integrable and the fourth is non-integrable.
\end{theorem*}

We also obtain the classification of invariant almost complex structures for the infinity family of flag manifolds $SU(3n)/S(U(n)\times U(n)\times U(n))$:
\begin{theorem*}
The family of generalized flag manifolds  $SU(3n)/S(U(n)\times U(n)\times U(n))$ has two invariant almost complex structures, up to conjugation and equivalence: one is an integrable structure and the other is non-integrable.
\end{theorem*}

With respect to other Lie groups, we obtain a partial classification for several generalized flag manifolds of the classical Lie groups $SO(n), Sp(n)$ and for the exceptional Lie group $G_2$, see sections \ref{sec1}, \ref{sec2}, \ref{sec3} and \ref{sec4}  for more details.


\section{Generalized flag manifolds and $k$-symmetric spaces}
\subsection{Generalized flag manifolds}
Let $\got{g}$ be a complex semi-simple Lie algebra and  $\got{h}$ be a Cartan sub-algebra of  $\got{g}$. Denote by  $\Pi$ the set of roots of the pair  $(\got{g},\got{h})$ and consider the decomposition
$$
\got{g}=\got{h}\oplus\sum_{\alpha\in \Pi}\got{g}_{\alpha},
$$ 
where $\got{g}_{\alpha}=\{X\in\got{g}:\forall H\in \got{h},[H,X]=\alpha(H)X \}$ is the complex root space (with complex dimension one). 

The Cartan-Killing form on $\mathfrak{g}$ is given by  
$$
\langle X,Y\rangle=\tr(\ad(X)\ad(Y))
$$ 
and its restriction to $\got{h}$ is non-degenerated. Given $\alpha\in \got{h}^{*}$, we define $H_{\alpha}$ by $\alpha(\cdot)=\langle H_{\alpha},\cdot\rangle$ and $\got{h}_{\mathbb{R}}=\spam_{\mathbb{R}}\{H_{\alpha}:\alpha\in \Pi\}$. 

We fix a Weyl basis of $\got{g}$, that is, $X_{\alpha}\in\got{g}_{\alpha}$ such that  if $\beta\neq -\alpha$, $ \langle X_{\alpha},X_{\beta}\rangle=0$,  $\langle X_{\alpha},X_{-\alpha}\rangle=1$ and $ [X_{\alpha},X_{\beta}]=m_{\alpha,\beta}X_{\alpha+\beta}$,
with $m_{\alpha,\beta}\in\mathbb{R}$, $m_{-\alpha,-\beta}=-m_{\alpha,\beta}$ and $m_{\alpha,\beta}=0$ if $\alpha+\beta$ is not a root. 

Let $\Pi^{+}\subset \Pi$ be a set of positive roots, $\Sigma$ the corresponding set of simple roots and $\Theta$ a subset of $\Sigma$. We fix the following notation: $\langle\Theta\rangle$ is the set of the roots spanned by  $\Theta$, $\Pi_{M}=\Pi\setminus\langle{\Theta}\rangle$ is the set of complementary roots and $\Pi_{M}^{+}$ is the set of complementary positive roots.

Let 
$$
\got{p}_{\Theta}=\got{h}\oplus
\sum_{\alpha\in\langle\Theta\rangle^{+}}\got{g}_{\alpha}\oplus
\sum_{\alpha\in\langle\Theta\rangle^{+}}\got{g}_{-\alpha}\oplus
\sum_{\beta\in \Pi_{M}^{+}}\got{g}_{\beta}
$$
be a parabolic sub-algebra $\got{g}$ determined by $\Theta$.

\begin{definition}
A generalized flag manifold $\mathbb{F}_{\Theta}$ associated to $\got{g}$ and $\Theta$ is the homogeneous space
$$
\mathbb{F}_{\Theta}=G/P_{\Theta},
$$
where $G$ is a complex Lie group with Lie algebra  $\got{g}$ and $P_{\Theta}$ is the normalizer of  $\got{p}_{\Theta}$ on $G$.
\end{definition}

Let $\got{u}$ be a compact real form of $\got{g}$ and $U=\exp(\mathfrak{u)}$. We have
$$
\got{u}=\spam_{\mathbb{R}}\{i\got{h}_{\mathbb{R}},A_{\alpha},iS_{\alpha}; \alpha\in \Pi\},
$$ 
where $A_{\alpha}=X_{\alpha}-X_{-\alpha}$ and $S_{\alpha}=X_{\alpha}+X_{-\alpha}$. 


Let $\got{k}_{\Theta}$ be the Lie algebra of $K_{\Theta}:=P_{\Theta}\cap U$. By construction $K_{\Theta}\subset U$ is the centralizer of a torus of $G$. We denote by  $\got{k}_{\Theta}^{\mathbb{C}}$ the complexified Lie algebra  $\got{k}_{\Theta}=\got{u}\cap \got{p}_{\Theta}$, that is, 
$$
\got{k}_{\Theta}^{\mathbb{C}}=\got{h}\oplus
\sum_{\alpha\in\langle\Theta\rangle^{+}}\got{g}_{\alpha}\oplus
\sum_{\alpha\in\langle\Theta\rangle^{+}}\got{g}_{-\alpha}.
$$

Since $U$ is compact and acts transitively on  $\mathbb{F}_{\Theta}$, we have 
$$
\mathbb{F}_{\Theta}=G/P_{\Theta}=U/(P_{\Theta}\cap U)=U/K_{\Theta}.
$$ 

If $\Theta=\emptyset$, we have
$$
\got{p}_{\Theta}=\got{p}=\got{h}\oplus
\sum_{\beta\in \Pi^{+}}\got{g}_{\beta}
$$ 
is a minimal parabolic sub-algebra (that is, a Borel sub-algebra) of $\got{g}$ and  
$$
\mathbb{F}=G/B=U/T
$$
is called {\it full flag manifold}, where $B$ is a Borel subgroup and $T=B\cap U$ is a maximal torus of  $U$. 

\subsection{The isotropy representation}

We denote by $x_{0}=eK_{\Theta}$ the origin of the flag manifold. 

Since $\mathbb{F}_{\Theta}=U/K_{\Theta}$ is a reductive homogeneous space, the Lie algebra of $U$ decomposes into 
$$
\got{u}=\got{k}_{\Theta}\oplus \got{m}_{\Theta} 
$$
with
$$
Ad(k)\got{m}_{\Theta}\subseteq \got{m}_{\Theta}, \forall k\in K_{\Theta}.
$$

The canonical projection $\pi:U\rightarrow\mathbb{F}_{\Theta}=U/K_{\Theta}$ induces an isomorphim between  $\got{m}_{\Theta}$ and $T_{x_{0}}\mathbb{F}_{\Theta}$. In some cases we write just $\mathfrak{m}$ instead $\mathfrak{m}_\Theta$.


The isotropy representation identifies with $Ad(k)\mid_{\got{m}_{\Theta}}:\got{m}_{\Theta}\longrightarrow\got{m}_{\Theta}$ and it is completely reducible, that is, 
$$
\got{m}_{\Theta}=\got{m}_{1}\oplus\got{m}_{2}\oplus\cdots\oplus\got{m}_{n},
$$
where each $\got{m}_i$ is an indecomposable and non-equivalent sub-representation (or equivalently, irreducible and non-equivalent $\mathfrak{k}_\Theta$-sub-modules).

The description of the irreducible sub-modules is given in the following way:

Consider the following sub-algebras:  $\got{h}=\got{h}^{\mathbb{C}}\cap \iota\got{k}$ and $\got{t}=Z(\got{k}^{\mathbb{C}})\cap \iota\got{h}$. Then $\got{k}^{\mathbb{C}}=\got{t}^{\mathbb{C}}\oplus \got{k'}^{\mathbb{C}}$, where $\got{k'}^{\mathbb{C}}$ is the semi-simple part of $\got{k}^{\mathbb{C}}$.

We consider the restriction map 
$$
\begin{array}{clcl}
\kappa:&\got{h}^{*}&\longrightarrow &\got{t}^{*}\\
&\alpha&\longmapsto&\alpha|_{\got{t}}.
\end{array}
$$

\begin{definition}\label{t-root}
The elements of $R_{T}:=\kappa(\Pi_{M})$ are called $T$-roots.
\end{definition}

\begin{theorem}[\cite{grego}]\label{142}
There exists a 1-1 correspondence between $T$-roots $\xi$ and irreducible $ad(\got{k}^{\mathbb{C}})$-modules $\got{m}_{\xi}$, given by
$$
R_{T}\ni\xi\longleftrightarrow m_{\xi}=\sum_{\kappa(\alpha)=\xi}\got{g}_{\alpha}.
$$
These sub-modules are non-equivalent as $\got{k}^{\mathbb{C}}$-modules.
\end{theorem}

Therefore a decomposition of $\got{m}^{\mathbb{C}}$ into irreducible $\ad(\got{k}^{\mathbb{C}})$-modules is given by
$$
\got{m}^{\mathbb{C}}=\sum_{\xi\in R_{T}}\got{m}_{\xi}.
$$

We observe that the complex conjugation $\tau$ of $\got{g}^{\mathbb{C}}$, interchanges the root spaces $\got{g}_{\alpha}$ and $\got{g}_{-\alpha}$, and consequently interchanges $\got{m}_{\xi}$ and $\got{m}_{-\xi}$.
We have the following decomposition
$$
\got{m}=\sum_{\xi\in R_{T}^{+}}(\got{m}_{\xi}+\got{m}_{-\xi})^{\tau},
$$
where $R_{T}^{+}=\kappa(R^{+})$ denote the set of positive $T$-roots, and $\got{n}^{\tau}$ denotes the set of the fixed points of  $\tau$ in a vector sub-space $\got{n}\subset \got{g}^{\mathbb{C}}$. 
\subsection{$k$-symmetric spaces}
Let $(M,g)$ be a Riemannian manifold and $x\in M$. An isometry of $(M,g)$ with isolated fixed point $x$ is called a symmetry of $(M,g)$ at $x$.

\begin{definition}
Assume that $(M,g)$ admits a set $\{s_{x}:x\in M\}$ of symmetries. We call $\{s_{x}:x\in M\}$ of a {\bf Riemannian $k$-symmetric structure} on $(M,g)$ if, for $x,y\in M$ we have:
$$
\begin{array}{c}
s_{x}\circ s_{y}=s_{z}\circ s_{x},\ \ \ (z=s_{x}(y))\\
(s_{x})^{k}=id,\ \ \ (s_{x})^{k}\neq id \ \ (l<k).
\end{array}
$$ 
Then $(M,g)$ with a Riemannian $k$-symmetric structure is called a $k$-symmetric space. 
\end{definition}

Let $M=G/K$ be a homogeneous space with origin $o=eK$ (trivial coset) and $g$ be a $G$-invariant Riemannian metric. We call the pair $(M,g)$ a Riemannian homogeneous space.

Given an automorphism $\theta$ of $G$ we define $G^{\Theta}$ to be the set of fixed points of $\theta$ and $G_{0}^{\Theta}$ the connected component of the identity.

\begin{proposition}\label{prop11}
Suppose that exist an automorphism $\Theta$ of $G$ such that  

\begin{itemize}

\item $G_{0}^{\Theta}\subseteq H\subseteq G^{\Theta}$;

\item $\Theta^{k}=1$ and $\Theta_{l}\neq0$ for any $l<k$;

\item Let $s$ be the transformation of $M$ defined by $\pi\circ \Theta=s\circ \pi$. Then $s$ preserves the metric at $o$.

\end{itemize}
Then $\{s_{x}=g\circ s \circ g^{-1}:x=g.o\in M\}$ define a Riemannian $k$-symmetric structure on $(M,g)$. 
\end{proposition}

\begin{remark}
We remark that if $k=2$, $G/K$ is called {\it symmetric space}. In general $k$-symmetric spaces are also called {\it generalized symmetric spaces}.
\end{remark}

The next Proposition, proved by Tojo in \cite{Tojo} (see also Burstall-Rawnsley \cite{B-R}), shows that every generalized flag manifold is a $k$-symmetric space. This result will be very useful in this work.

\begin{proposition}[\cite{Tojo}]\label{flagksim}
Let $M=G/K$ be a generalized flag manifold. Then $G/K$ admits an invariant complex structure $J$ such that $(G/K,J,g)$ has an Hermitian $m$-symmetric structure for each $G$-invariant Riemannian metric $g$.
\end{proposition}

The next Lemma gives an important information about the cohomology of $k$-symmetric spaces. 
\begin{lemma}[\cite{Kote}]\label{rel2}
Let $G/K$ and $G/L$ $k$-symmetric spaces with $G$ a compact simple Lie group, $K\subset L$ and $\rk(K)=\rk(L)$. Then the homogeneous fibration $G/K\longrightarrow G/L$ with fiber $L/K$ has the following property: the restriction map from the total space to the fiber is surjective in cohomology ring. 
\end{lemma}

\section{Invariant almost complex structures}

\begin{definition}
An almost complex structure on the flag manifold  $\mathbb{F}_{\Theta}$ is a tensor  $J^{\Theta}$  such that for each point  $x\in \mathbb{F}_{\Theta}$ we have an endomorphism  $J^{\Theta}:T_{x}\mathbb{F}_{\Theta}\longrightarrow T_{x}\mathbb{F}_{\Theta}$ satisfying $(J^{\Theta})^{2}=-Id$.
\end{definition}

\begin{definition}
An $U$-invariant almost complex structure  $J^{\Theta}$ on  $\mathbb{F}_{\Theta}=U/K_{\Theta}$ is given by  
$$
J_{ux}^{\Theta}=dE_{u}J_{ux}^{\Theta}dE_{u^{-1}},\ \mbox{ for each } \ u\in U,
$$ 
where  $dE_{u}:T(U/K_{\Theta})\rightarrow T(U/K_{\Theta})$ is the differential of the left action at $u\in U$, that is, for each  $X\in T_{x}(U/K_{\Theta})$ we have 
$$
dE_{u}J_{x}^{\Theta}X=J_{ux}^{\Theta}dE_{u}X.
$$
\end{definition}

A $U$-invariant almost complex structure $J^{\Theta}$ on $\mathbb{F}_{\Theta}$ is completely determined by $J^{\Theta}:\got{m}_{\Theta}\rightarrow\got{m}_{\Theta}$, where $\mathfrak{m}_\Theta$ is the tangent space at the origin of $\mathbb{F}_{\Theta}$. 

\begin{proposition}
There is a 1-1 correspondence between invariant almost complex structures $J^{\Theta}$ and linear maps  $J_{x_{0}}^{\Theta}$ on $T_{x_{0}}\mathbb{F}_{\Theta}$ that commute with the isotropy representation, that is,
$$
Ad^{U/K_{\Theta}}(k)J_{x_{0}}^{\Theta}=J_{x_{0}}^{\Theta}Ad^{U/K_{\Theta}}(k), \ \mbox{ for all } \ k\in K_{\Theta}. 
$$
\end{proposition}

Therefore an invariant almost complex structure  $J^{\Theta}$ satisfies $(J^{\Theta})^{2}=-1$ and commutes with the adjoint action of  $\got{k}_{\Theta}$ on  $\got{m}_{\Theta}$. Moreover  $J^{\Theta}(\got{g}_{\alpha})=\got{g}_{\alpha}$, for each $\alpha\in\Pi$. The eigenvalues  of  $J^{\Theta}$ are  $\pm i$ and the eigenvectors are $X_{\alpha}$, $\alpha\in\Pi$. 

Hence  $J^{\Theta}(X_{\alpha})=i \varepsilon_{\alpha}X_{\alpha}$, being $\varepsilon_{\alpha}=\pm 1$ with $\varepsilon_{\alpha}=-\varepsilon_{-\alpha}$. In this way we obtain a description of an invariant almost complex structure on $\mathbb{F}_{\Theta}$: they are completely described by a set of signs 
$$
\{\varepsilon_{\alpha}, \, \mbox{ with } \,  \varepsilon_{\alpha}=\pm 1, \  \alpha\in\Pi\setminus\langle\Theta\rangle,  \mbox{ such  that }  \varepsilon_{\alpha}=-\varepsilon_{-\alpha}\}.
$$

By simplicity we will denote an invariant almost complex structure just by $J$. As usual, we denote by $T^ {1,0}M$ the eigenspace of $J$ associated to the eigenvalue $+i$.

\begin{definition}
An almost complex structure $J$ on a differential manifold $M$ is said a {\it complex structure} (or integrable) if its distribution $T^ {1,0}M$ is integrable. The pair $(M,J)$ is called a {\it complex manifold}.

Two complex structures  $J_{1}$ and $J_{2}$ on  $M$ are equivalent if the complex manifolds $(M,J_{1})$ and $(M,J_{2})$ are biholomorphic, that is, if there exists a holomorphic map $\phi:(M,J_1) \to (M,J_2)$ with holomorphic inverse.
\end{definition}

The next three results due to  Borel-Hirzebruch concerning to invariant almost complex structure are very useful in this work.


\begin{proposition}[\cite{opa},13.4]\label{numest}
Let $M=G/K$ be an almost complex manifold and  $\got{m}=\got{m}_{1}\oplus\got{m}_{2}\oplus\cdots\oplus\got{m}_{s}$ be a decomposition of the tangent space at the origin into irreducible and non-equivalent sub-modules of the isotropy representation. Then $M$ admits $2^{s}$ invariant almost complex structures.  

If we identify conjugated structures, then $M$ admits $2^{s-1}$ invariant almost complex structures, up to conjugation.
\end{proposition}

A root system is said to be {\bf closed} if, for every complementary roots $\alpha$ and $\beta$ such that  $\alpha+\beta$ is a root, $\alpha+\beta$ is again a complementary root.

The next lemma provides a useful criteria to determine when an invariant almost complex structure is integrable.

\begin{lemma}[\cite{opa}, 13]\label{order}
An invariant almost complex structure is integrable if there exist an order on the coordinates of the Cartan sub-algebra such that the corresponding complementary root system is positive and closed.

\end{lemma}

The next proposition tell us when two invariant complex structures on flag manifolds are equivalent.

\begin{proposition}[\cite{opa}, 13]\label{si}
Let $J_1$ and $J_2$ be two invariant complex structures on $G/K$. Assume that there exist an automorphism of the Cartan sub-algebra  $\mathfrak{h}$ of $\mathfrak{g}$ that send the root system associated to $J_1$ onto the root system associated to $J_2$ and keep fixed the root system of $K$. 

Then these two invariant complex structures are equivalent under an automorphism of $G$ that keep $K$ fixed. 
\end{proposition}

\section{Formality in the sense of Sullivan}

A commutative differential graduated algebra $(A,d)$ is called {\it formal} if it is weakly equivalent to the cohomology algebra $(H(A,\mathbb{Q}),0)$, that is, there exist a sequence of quasi-isomorphism (morphism of commutative differential graduated algebras that induce isomorphism in cohomology) in the following way
$$
(A,d)\stackrel{\simeq}{\longrightarrow}\cdots\stackrel{\simeq}{\longrightarrow}\cdots\stackrel{\simeq}{\longrightarrow}\cdots\stackrel{\simeq}{\longrightarrow}(H(A,\mathbb{Q}),0).
$$ 

A differential manifold is formal (in the sense of Sullivan) if their de Rham algebra of differential forms and the cohomology algebra (with the zero-differential) are weakly equivalent. For us, {\it formal manifolds}  mean {\it formal in the sense of Sullivan}.

An important source of examples of formality in the sense of Sullivan are the $k$-symmetric spaces.
\begin{theorem}[\cite{Kote}, teo. 7]
Each $k$-symmetric space of a compact Lie groups is formal in the sense of Sullivan. 
\end{theorem}

Since generalized flag manifolds are $k$-symmetric (see Proposition \ref{flagksim}), every generalized flag manifolds are formal in the sense of Sullivan.

\section{Geometric formality}

Let $M$ be a compact oriented differential manifold, $g$ a Riemannian metric on $M$ and $\Omega^{k}(M)$ the complex of degree $k$ differential forms on $M$. Therefore the de Rham complex is the following sequence of differential operators
$$
0\rightarrow \Omega^{0}(M)\stackrel{d_{0}}{\rightarrow} \Omega^{1}(M)\stackrel{d_{1}}{\rightarrow}...\stackrel{d_{n-1}}{\rightarrow} \Omega^{n}(M)\stackrel{d_{n}}{\rightarrow} 0,
$$
where $d_{k}$ denote the exterior derivative on $\Omega^{k}(M)$. The de Rham cohomology is the sequence of vector spaces defined by
$$
H^{k}(M,\mathbb{R})=\dfrac{ker(d_{k})}{Im(d_{k-1})}.
$$ 
We define the adjoint of the exterior derivative $\delta$ called co-differential as follows:

for each $\alpha\in\Omega^{k}(M)$ and $\beta\in\Omega^{k+1}(M)$, $\delta$ is given by
$$
\langle d\alpha,\beta\rangle_{k+1}=\langle\alpha,\delta\beta\rangle_{k},
$$
where $\langle \ ,\ \rangle$ is the metric induced in $\Omega^{k}(M)$. The Laplacian acting on forms is defined by 
$\Delta=d\delta+\delta d$ and is called Laplace-Beltrami operator.

\begin{definition}
The space of $k$-harmonic forms is defined by $$\mathcal{H}^{k}_{\Delta}(M)=\{\alpha\in \Omega^{k}(M);\Delta(\alpha)=0\}.$$
\end{definition}

We remark that a form $\omega$ is harmonic if it is closed and co-closed, that is, $d\omega=0$ and $\delta\omega=0$.

\begin{theorem}(Hodge)\label{Hod}
Each de Rham cohomology class of a compact oriented differentiable manifold has unique harmonic representative.
\end{theorem}

Given a Riemannian manifold $(M,g)$, the wedge product of harmonic forms is not harmonic in general. 

Sullivan in \cite{Sul} proved that there exists topological obstructions to a Riemannian manifold admits a metric such that the wedge product of harmonic form is harmonic. This motivated the following definition.

\begin{definition}[Kotschick, \cite{Kots}]
A Riemannian metric is called {\it (metrically) formal} if every wedge product of harmonic forms is harmonic. A smooth manifold is called {\it geometrically formal} if it admit a formal Riemannian metric. 
\end{definition}

Examples of geometrically formal manifolds are the compact globally symmetric spaces and spheres. The cartesian product of geometrically formal manifolds (with the product metric) is again geometrically formal.


Other examples of geometrically formal manifolds are (see \cite{Ko1} for details): $\mathbb{H}P^{2}=Sp(3)/(Sp(2)\times Sp(1))$, $\mathbb{O}P^{2}=F_{4}/Spin(9)$, $G_{2}/SO(4)$, real Stiefel manifolds $V_{4}(\mathbb{R}^{2n+1})=SO(2n+1)/SO(2n-3), n\geq3$ and $V_{3}(\mathbb{R}^{2n})=SO(2n)/SO(2n-3), n\geq3$.

Geometric formality implies in the formality in the sense of Sullivan. In \cite{Kote} one can find examples of manifolds that are formal in the sense of Sullivan but not geometrically formal. 

Clearly the problem of finding formal metric is a question in Riemannian geometry. On the other hand the existence of a formal metric implies in a restriction on the topology of  $M$. 

In this work we will proceed a careful analysis in the cohomology ring on a given homogeneous space in order to find topological obstruction to existence of the formal metric.

From this topological approach we cite \cite{Kote}. In this work Kotschick and Terzi\'c proved that full flags manifolds of the classical Lie groups $SU(n)$, $SO(2n)$, $SO(2n)$, $Sp(n)$ and the full flag manifold of the exceptional Lie group $G_2$, do not admit any formal metric. 

We will prove that the full flag manifolds of the exceptional Lie groups $E_6, E_7,E_8$ and $F_4$ do not admit any formal metric.

\section{Geometric formality for full flag manifolds}


Kotschick and Terzi\'c have shown in \cite{Kote} that full flag manifolds associated to the simple Lie groups $SU(n)$, $SO(2n+1)$, $Sp(n)$, $SO(2n)$ and $G_{2}$ are not geometrically formal. In this section we will show that the full flag manifold associated to $F_{4}$, $E_{6}$, $E_{7}$ and $E_{8}$ is not geometrically formal too. This will give a complete understanding of geometric formality on full flag manifolds of simple Lie groups, according the Cartan-Killing classification.

Let us recall some results about the cohomology structure of full flag manifolds of classical Lie groups.

\begin{proposition}[\cite{Kote}]\label{p1}
The class represented by 
\begin{equation}\label{a8}
x_{1}^{\alpha_{1}}x_{2}^{\alpha_{2}}\cdots x_{n}^{\alpha_{n}},\ \ \ 0\leq\alpha_{i}\leq i, \ \ 1\leq i\leq n,
\end{equation}
form a basis for the cohomology of $SU(n+1)/T^{n}$ as a vector space. The multiplicative relations between the $x_{1},\cdots,x_{n}$ are given by:
\begin{equation}\label{a9}
\sum_{i_{1}+\cdots+i_{p}=n-p+2}x_{n-p+1}^{i_{p-1}}x_{n-p+2}^{i_{p-1}}\cdots x_{n-1}^{i_{2}}x_{n}^{i_{1}}=0, \ \ 1\leq p\leq n.
\end{equation}
\end{proposition}

\begin{lemma}[\cite{Kote}]\label{p2} Let $M$ be a differentiable  manifold of dimension  $n^{2}+n$, with $n\geq 2$. Suppose there are  $n$ closed 2-forms, $x_{1},\cdots, x_{n}$ on $M$ satisfying relations (\ref{a9}). Then $\omega=x_{1}\wedge x_{2}^{2}\wedge \cdots \wedge x_{n}^{n}$ vanishes identically. In particular, $\omega$ is not a volume form on $M$.
\end{lemma}

\begin{proposition}[\cite{Kote}]\label{rel3}
The classes represented by 
$$
x_{1}^{\alpha_{1}}x_{2}^{\alpha_{2}}\cdots x_{n}^{\alpha_{n}},\ \ \ 0\leq\alpha_{i}\leq 2i-1, \ \ 1\leq i\leq n,
$$
form a vector space basis for the cohomology ring of $Spin(2n+1)/T^{n}$ and $Sp(n)/T^{n}$. The multiplicative relations between $x_{1},\cdots,x_{n}$ are given by 
\begin{equation}\label{rel}
\sum_{i_{1}+\cdots+i_{p}=n-p+1}x_{n-p+1}^{2{i_{p}}}x_{n-p+2}^{2{i_{p-1}}}\cdots x_{n-1}^{2{i_{2}}}x_{n}^{2{i_{1}}}=0,\ \ 1\leq p\leq n.
\end{equation}
\end{proposition}

\begin{lemma}[\cite{Kote}]\label{rel1}
Let $M$ be differentiable manifold of dimension $2n^{2}$, with $n\geq2$. Suppose that there are $n$ closed $2$-forms $x_{1},\cdots,x_{n}$ on $M$ satisfying the relations (\ref{rel}). Then $\omega=x_{1}\wedge x_{2}^{3}\wedge\cdots\wedge x_{n}^{2n-1}$ vanishes identically. In particular, $\omega$ is not a volume form on $M$.
\end{lemma}

We now proof the main theorem of this section.

\begin{theorem}\label{teoprinc}
The full flag manifolds $F_{4}/T$, $E_{6}/T$, $E_{7}/T$ and $E_{8}/T$ are not geometrically formal, where $T$ represents the maximal torus for each corresponding Lie group.
\end{theorem}

\begin{proof}
We will proof the Theorem analyzing case by case.
\begin{enumerate}

\item In this way, we start proving that the flag manifold $F_{4}/T$ is not geometrically formal.

Let us consider the following fibration
$$
Spin(9)/T\longrightarrow F_{4}/T \longrightarrow F_{4}/Spin(9).
$$
According to the Leray-Hirsch's Theorem we have 
$$
H^{*}(F_{4}/T,\mathbb{R})\cong H^{*}(Spin(9)/T,\mathbb{R})\otimes H^{*}(F_{4}/Spin(9),\mathbb{R}).
$$
Since  $F_{4}/T$ and $Spin(9)/T$ are $k$-symmetric spaces, Lemma \ref{rel2} implies that the restriction to the fiber is surjective in real cohomology.

Suppose that $F_{4}/T$ is geometrically formal.

We now use the basis $x_{1},x_{2},x_{3},x_{4}$ for the cohomology of $Spin(9)/T$ given in the Proposition \ref{rel3}. Abusing of the notation we denote by  $x_{1},x_{2},x_{3},x_{4}$ the harmonic representatives of the cohomology classes $x_{1},x_{2},x_{3},x_{4}$ on $F_4/T$ with respect to a formal metric. Therefore the relations (\ref{rel}) hold for these harmonic forms. 


If we restrict these forms to the fiber, using Lemma \ref{rel1} we have that $x_{1}\wedge x_{2}^{3}\wedge x_{3}^{5}\wedge x_{4}^{7}$ vanishes identically and this contradicts the Lemma \ref{rel2}.

\item The flag manifold $E_{6}/T$ is not geometrically formal.

We consider the following fibration
$$
SU(6)/T^{5}\longrightarrow E_{6}/T^{6} \longrightarrow E_{6}/SU(6)\times U(1).
$$

Using Leray-Hirsch's Theorem we have 
$$
H^{*}(E_{6}/T,\mathbb{R})\cong H^{*}(SU(6)/T,\mathbb{R})\otimes H^{*}(E_{6}/SU(6)\times U(1),\mathbb{R}).
$$
Since $E_{8}/T$ and $E_{6}/SU(6)\times U(1)$ are $k$-symmetric spaces,  Lemma \ref{rel2} implies that the restriction to the fiber is surjective in real cohomology. 

Suppose that $E_{6}/T$ is geometrically formal.

We use the basis $x_{1},\cdots,x_{5}$ for the cohomology of $SU(6)/T$ given in the Proposition \ref{p1}. Abusing of the notation we denote by $x_{1},\cdots,x_{5}$ the harmonic representatives of the cohomology classes $x_{1},\cdots,x_{5}$ on $E_6/T$ with respect to a formal metric. Therefore the relations (\ref{a9}) hold for these harmonic forms. 

If we restrict these forms to the fiber $SU(6)/T$, Lemma \ref{p2} implies that the form $x_{1}\wedge x_{2}^{2}\wedge x_{3}^{3}\wedge x_{4}^4\wedge x_{5}^{5}$ vanishes identically and this contradicts the Lemma \ref{rel2}.

\item The flag manifold $E_{7}/T$ is not geometrically formal.

Let us consider the following fibration
$$
{SU(7)}/{T^{6}}\longrightarrow {E_{7}}/{T^{7}} \longrightarrow {E_{7}}/{SU(7)\times U(1)}.
$$
By Leray-Hirsch's Theorem we have 
$$
H^{*}(E_{7}/T^{7},\mathbb{R})\cong H^{*}(SU(7)/T^6,\mathbb{R})\otimes H^{*}(E_{7}/SU(7)\times U(1),\mathbb{R}).
$$
Since $E_{7}/T^{7}$ and $E_{7}/SU(7)\times U(1)$ are $k$-symmetric spaces Lemma \ref{rel2} implies that the restriction to the fiber is surjective in real cohomology.

Suppose that $E_{7}/T$ is geometrically formal.

We use the basis $x_{1},\cdots,x_{6}$ for the cohomology of $SU(7)/T$ given in the Proposition \ref{p1}. Abusing of the notation we denote by  $x_{1},\cdots,x_{6}$ the harmonic representatives of the cohomology classes $x_{1},\cdots,x_{6}$ on $E_7/T$ with respect to a formal metric. Therefore the relations (\ref{a9}) hold for this harmonic forms.  

If we restrict these forms to the fiber $SU(7)/T$, Lemma \ref{p2} implies that the form $x_{1}\wedge x_{2}^{2}\wedge x_{3}^{3}\wedge\cdots \wedge x_{6}^{6}$ vanishes identically and this fact contradicts the Lemma \ref{rel2}.

\item The flag manifold $E_{8}/T$ is not geometrically formal.

Let us consider the following fibration
$$
{SU(8)}/{T^{7}}\longrightarrow {E_{8}}/{T^{8}} \longrightarrow {E_{8}}/{SU(8)\times U(1)}.
$$
By Leray-Hirsch's Theorem we have 
$$
H^{*}(E_{8}/T,\mathbb{R})\cong H^{*}(SU(8)/T,\mathbb{R})\otimes H^{*}(E_{8}/SU(8)\times U(1),\mathbb{R}).
$$
Since $E_{8}/T$ and $E_{8}/SU(8)\times U(1)$ are $k$-symmetric spaces Lemma \ref{rel2} implies that the restriction to the fiber is surjective in real cohomology.

Suppose that $E_{8}/T$ is geometrically formal.

We use the basis $x_{1},x_{2},\cdots,x_{7}$ for the cohomology of $SU(8)/T$  given in the Proposition \ref{p1}. Abusing of the notation we denote by $x_{1},x_{2},\cdots,x_{7}$ the harmonic representatives of the cohomology classes $x_{1},\cdots,x_{7}$ on $E_8/T$  with respect to a formal metric. Therefore the relations (\ref{a9}) hold for these harmonic forms.  

If we restrict these forms to the fiber $SU(8)/T$, Lemma \ref{p2} implies that the form  $x_{1}\wedge x_{2}^{2}\wedge x_{3}^{3}\wedge\cdots \wedge x_{7}^{7}$ vanish identically and this fact contradicts the Lemma \ref{rel2}.  
\end{enumerate}
\end{proof}

In \cite{Kote}, Kotschick and Terzi\'c proved that full flag manifolds of classical Lie groups and the full flag manifold of the exceptional Lie group $G_2$ are not geometrically formal. Therefore, together with the results of Kotschick and Terzi\'c we have completed the list of full flag manifolds of compact simple Lie groups. We summarize this fact in the next Corollary: 

\begin{corollary}
Let $G$ be a connected, compact, simple Lie group, $T$ a maximal torus in $G$. The correspondent full flag manifold $G/T$ is not geometrically formal.  
\end{corollary}


\section{Non geometric formality on other homogeneous space}\label{sec-nonf}
In this section we proof the non-geometric formality for several homogeneous space, including homogeneous space of exceptional Lie groups.

\begin{proposition}
The generalized flag manifold $\mathbb{F}_{D}(3;1,2)=SO(6)/U(2)\times U(1)$ is not geometrically formal.
\end{proposition}
\begin{proof}
The cohomology ring of $\mathbb{F}_{D}(3;1,2)$ is given by
$$
H^{*}(\mathbb{F}_{D}(3;1,2),\mathbb{R})=\dfrac{\mathbb{R}[x,y,z]}{\langle s_{1},s_{2},e_{3}\rangle},
$$
where $s_{1}=x^2+y^2+z^2$, $s_{2}=x^4+y^4+z^4$ and $e_{3}=xyz$.

A Gr\"obner basis for the ideal $\langle s_{1},s_{2},e_{3}\rangle$ is
$$
\begin{array}{lccccccccl}
b_{1}=x^2+y^2+z^2,&&&&&&&&&b_{4}=y^3z+yz^3,\\
b_{2}=xyz,&&&&&&&&&b_{5}=z^5.\\
b_{3}=y^4+y^2z^2+z^4,
\end{array}
$$
Therefore we can consider the cohomology ring represented in the following way
$$
H^{*}(\mathbb{F}_{D}(3;1,2),\mathbb{R})=\dfrac{\mathbb{R}[x,y,z]}{\langle b_{1},b_{2},b_{3},b_{4},b_{5}\rangle}.
$$

Now $y^{2}z^{3}-yz^{4}$ generates the top-dimensional class $H^{10}(\mathbb{F}_{D}(3;1,2),\mathbb{R})$.

Suppose $\mathbb{F}_{D}(3;1,2)$ is geometrically formal.

By Theorem \ref{Hod} and abusing of the notation we still denote by $x,y,z$ the harmonic representatives of the cohomology classes $x,y,z\in H^{2}(\mathbb{F}_{D}(3;1,2),\mathbb{R})$.

Using the relations $b_{i}'s$ and geometric formality we have that $y^{2}z^{3}-yz^{4}$ is a volume form on $\mathbb{F}_{D}(3;1,2)$.

Given a 2-form $\alpha$, we denote by $N_\alpha =\{ v \in TF_D(3;1,2): i_v\alpha=0   \}$  - the kernel distribution of $\alpha$. Since the generators (2-forms) $y,z$ satisfy $z^{5}=y^{5}=0$ and $ \dim F_D(3;1,2)=10$ it implies that the kernel distribution $N_y$ and $N_z$ has rank at least 2. Therefore we can choose locally linearly independent vector fields $v\in N_{y}$ and $w\in N_{z}$.
Note that by relation $b_{4}$ we have
\begin{equation}\label{a13}
y^{3}z+yz^{3}=0\Longrightarrow y^{3}z^{2}+yz^{4}=0\Longrightarrow y^{3}z^{2}=-yz^{4}.
\end{equation}
Moreover, for any $v\in N_y$ and $w\in N_z$ we have
$$
\begin{array}{ccl}
i_{w}(i_{v}(y^{4}+y^{2}z^{2}+z^{4}))&=&i_{w}(y^{2}\wedge(i_{v}z^{2})+i_{v}z^{4})\\
&=&(i_{w}y^{2})\wedge(i_{v}z^{2})+y^{2}\wedge i_{w}(i_{v}z^{2})+i_{w}(i_{v}z^{4})\\
&=&4(i_{w}y)\wedge y\wedge(i_{v}z)\wedge z,
\end{array}
$$
and we get 
\begin{equation}\label{a17}
(i_{w}y)\wedge y\wedge(i_{v}z)\wedge z=0.
\end{equation}

Using (\ref{a13}) and (\ref{a17}) we obtain
$$
\begin{array}{ccl}
i_{w}i_{v}(y^{3}z^{3}-yz^{4})&=&i_{w}i_{v}(y^{2}z^{3}+y^{3}z^{2})\\&=&
i_{w}(y^{2}\wedge(i_{v}z^{3})+y^{3}\wedge(i_{v}z^{2}))\\&=&
(i_{w}y^{2})\wedge(i_{v}z^{3})+y^{2}\wedge i_{w}(i_{v}z^{3})+(i_{w}y^{3})\wedge(i_{v}z^{2})+y^{3}\wedge i_{w}(i_{v}z^{2})\\&=&
2(i_{w}y)\wedge y\wedge 3(i_{v}z)\wedge z^{2}+3(i_{w}y)\wedge y^{2}\wedge 2(i_{v}z)\wedge z\\&=&
6((i_{w}y)\wedge y\wedge (i_{v}z)\wedge z)\wedge z+
6((i_{w}y)\wedge y\wedge (i_{v}z)\wedge z)\wedge y\\&=&
0,
\end{array}
$$ 
and this contradicts the fact that $y^{2}z^{3}-yz^{4}=y^{2}z^{3}+y^{3}z^{2}$ is a volume form; therefore  $SO(6)/U(1)\times U(2)$ can not be geometrically formal.
\end{proof}

The next Proposition describes the cohomology ring of the generalized flag manifold $\mathbb{F}(n+2;n,1,1)=SU(n+2)/{S(U(n)\times U(1)\times U(1))}$.

\begin{proposition}[\cite{princ}]\label{xx}
The cohomology ring of $\mathbb{F}(n+2;n,1,1)$ is generated by two elements $x$ and $y$ of degree $2$ such that 
\begin{equation}
x^{n+2}=0
\end{equation} 
and 
\begin{equation}\label{rel11}
\dfrac{(x+y)^{n+2}-x^{n+2}}{y}=0.
\end{equation}
 The generators $x$ and $y$ can be described as follow: consider the fibration $$p:\mathbb{F}(n+2;n,1,1) \rightarrow \mathbb{CP}^{n+1}$$ given by the projectivization of the tangent bundle
of $\mathbb{CP}^{n+1}$. Then $x = p^{*}(H)$ denote the pullback of the hyperplane class and $y$ denote the tautological class on the total space, restricting to the hyperplane class on every fiber.

Moreover, the relation \ref{rel11} can be rewritten as  
\begin{equation}\label{aaa}
y^{n+1} + c_{1}y^{n} +\cdots + c_{n+1} = 0,
\end{equation}
where the $c_{i}$ are the pullbacks to the total space of the Chern classes of the base.
\end{proposition}



Kotschick and Terzi\'c have proved that the family $\mathbb{F}(n+2;n,1,1)$ is not geometrically formal (see \cite{princ}).
In the next Proposition we analyse the geometric formality of the family of flag manifolds $\mathbb{F}(n+3;n,1,1,1)$.

\begin{proposition}\label{familia1}
The family of generalized flag manifolds $\mathbb{F}(n+3;n,1,1,1)=SU(n+3)/S(U(n)\times U(1)^{3})$ are not geometrically formal.
\end{proposition}
\begin{proof}
Consider the following fibration
$$
\mathbb{F}(n+2;n,1,1)\longrightarrow \mathbb{F}(n+3;n,1,1,1)
\longrightarrow \mathbb{CP}^{n+2}.
$$

Since the basis and the total space of such fibration are $k$-symmetric spaces, Lemma \ref{rel2} tell us that all cohomology classes of $\mathbb{F}(n+2;n,1,1)$ are restrictions of cohomology class of $\mathbb{F}(n+3;n,1,1,1)$. 
 
According to Proposition \ref{xx}, $H^{*}(\mathbb{F}(n+2;n,1,1),\mathbb{R})$ is generated by two elements $x,y$ in degree $2$ with relations 
\begin{equation}\label{a1}
x^{n+2}=0 \ \ \ \ \mbox{and} \ \ \ \ \ \dfrac{(x+y)^{n+2}-x^{n+2}}{y}=0.
\end{equation}

According to Leray-Hirsch's Theorem, $H^{*}(\mathbb{F}(n+3;n,1,1,1),\mathbb{R})$ is a $H^{*}(\mathbb{CP}^{n+2},\mathbb{R})$-module generated by $x,y$. 

We use the basis of $H^*(\mathbb{F}(n+2;n,1,1),\mathbb{R})$ given by $x$ and $z=x+y$ in $H^{2}(\mathbb{F}(n+2;n,1,1),\mathbb{R})$.

Suppose that $\mathbb{F}(n+3;n,1,1,1)$ is geometrically formal.

Using Theorem $\ref{Hod}$ we can identify $x,z$ with their harmonic representatives (we use the same $y,z$ to denote the harmonic forms). On $\mathbb{F}(n+3,n,1,1,1)$, $x$ and $z$ satisfy $x^{n+2}=z^{n+2}=0$ and $x^{n+1}\neq0\neq z^{n+1}$, that is, $\rk(x)=\rk(z)=2n+2$.

The rank of the kernel of $x$ is $\dim(\mathbb{F}(n+3;n,1,1,1))-rk(x)=6n+6-2n-2=4n+4$. Analogously the rank of the kernel of $z$ is $4n+4$. 

Since the codimension of the fiber is $$\dim(\mathbb{F}(n+3;n,1,1,1))-\dim(\mathbb{F}(n+2;n,1,1))=6n+6-4n-2=2n+4,$$ the restriction of $x$ and $z$ to the fiber has a kernel of rank $2n$. 

Now rewritten \ref{aaa} in terms of  $x$ and $z$ we obtain 
\begin{equation}\label{eqc1}
z^{n+1}+xz^{n}+x^{2}z^{n-1}+\cdots+x^{n}z^{n}+x^{n+1}=0.
\end{equation}
Contacting the equation \ref{eqc1} with a local basis for the kernel of $x$,  $\{v_{1},\cdots,v_{2n}\}$ we have
\begin{equation}\label{eqc2}
i_{v_{1}}\cdots i_{v_{2n}}z^{n+1}+x\wedge i_{v_{1}}\cdots i_{v_{2n}}z^{n}=0.
\end{equation}
Now contracting \ref{eqc2} with$w$ in the kernel of $z$  we have
\begin{equation}
i_{w}x\wedge i_{v_{1}}\cdots i_{v_{2n}}z^{n}=0.
\end{equation}

and therefore 
\begin{equation}
i_{w}i_{v_{1}}\cdots i_{v_{2n}}x^{n+1}z^{n}=0.
\end{equation}

Hence the restriction of $x^{n+1}z^{n}$ to $\mathbb{F}(n+2;n,1,1)$ vanishes identically and this contradicts the Lemma \ref{rel2}, finishing the proof.
\end{proof}



Our next result shows that a large class of $SU(n)$-generalized flag manifolds do not admit any formal metrics. We introduce the following notation: 
$$\mathbb{FL}(n)=\dfrac{SU(n)}{S(U(n_{1})\times U(n_{2})\times \cdots \times U(n_{k})\times U(1)^{m})}, \ n=n_{1}+\cdots +n_{k}+m,\ m\neq0 \ \normalfont{ and } \ m\geq2,$$ 
and
$$
B(n)=\dfrac{SU(n)}{S(U(n_{1}+2)\times U(n_{2})\times \cdots \times U(n_{k})\times U(1)^{m-2})}.
$$ 

\begin{theorem}
The family $\mathbb{FL}(n)$ is not geometrically formal.
\end{theorem}
\begin{proof}
The proof is very similar to the proof of Proposition \ref{familia1}. Assume that $\mathbb{FL}(n)$ is geometrically formal. Now consider the fibration $$
\mathbb{F}(n+2;n,1,1)\longrightarrow \mathbb{FL}(n)\longrightarrow B(n)$$ and use the cohomology description of the fiber given by Proposition \ref{xx} and  Leray-Hirsch's Theorem in order to get an contradiction. We omit the details.

\end{proof}

Actually, we can produce several examples of non-geometrically formal examples using the same techniques above (description of the cohomology of fiber, Leray-Hirsch's Theorem, etc.). For instance, we can use the fibration 
\begin{equation}\label{fibration1}
\dfrac{SU(3)}{T^{2}}\longrightarrow \dfrac{E_{6}}{SU(3)\times U(1)^{4}}\longrightarrow \dfrac{E_{6}}{SU(3)^{2}\times U(1)^{2}}
\end{equation}
in order to proof that $E_{6}/(SU(3)\times U(1)^{4})$  is not geometrically formal. We summarize this in the next result, exhibiting a list of homogeneous space that are not geometrically formal (including several examples of homogeneous space of exceptional Lie groups).

\begin{theorem}
The homogeneous spaces listed in the Table \ref{non-formal} are not geometrically formal.
\end{theorem}
\begin{proof}
We will give the details just for the manifold $E_{6}/(SU(3)\times U(1)^{4})$. Consider the fibration (\ref{fibration1}). Since the basis and the total space of such fibration are $k$-symmetric spaces, Lemma \ref{rel2} tell us that all cohomology classes of $SU(3)/T^2$ are restrictions of cohomology class of $E_{6}/(SU(3)\times U(1)^{4})$. 

By Leray-Hirsch's Theorem we have:
$$
H^{*}\left(\dfrac{E_{6}}{SU(3)\times U(1)^{4}},\mathbb{R}\right)\cong H^{*}\left(\dfrac{SU(3)}{T^{2}},\mathbb{R}\right)\otimes H^{*}\left(\dfrac{E_{6}}{SU(3)^{2}\times U(1)^{2}},\mathbb{R}\right).
$$ 

Consider the cohomology ring of  $SU(3)/T^{2}$ given by Proposition \ref{p1}. Suppose that $E_{6}/(SU(3)\times U(1)^{4})$ is geometrically formal. Using Theorem $\ref{Hod}$ we can identify $x_1,x_2$ with their harmonic representatives (we use the same $x_1,x_2$ to denote the harmonic forms with respect to the formal metric). Therefore the relations (\ref{a9}) hold for these harmonic forms. If we restrict these forms to the fiber $SU(3)/T^2$, Lemma \ref{p2} implies that the form  $x_{1}\wedge x_{2}^{2}$ vanish identically and this fact contradicts the Lemma \ref{rel2}.  

\end{proof}

\begin{table}
\caption{Homogeneous spaces non-geometrically formal.}
\label{non-formal}
\begin{tabular}{|c|c|}
\hline
$SO(2n+5)/(U(n)\times U(1)^{2})$&
$Sp(2n+2)/(U(n)\times U(1)^{2})$\\
\hline
$E_{6}/(SU(3)\times U(1)^{4})$&
$E_{6}/(SU(2)\times U(1)^{5})$\\
\hline
$E_{6}/(SU(3)\times SU(2)\times U(1)^{3})$&
$E_{6}/(SU(2)^{2}\times U(1)^{4})$\\
\hline
$E_{7}/(SU(3)\times U(1)^{5})$&
$E_{7}/(SU(4)\times U(1)^{3})$\\
\hline
$E_{7}/(SU(2)^{2}\times U(1)^{5})$&
$E_{7}/(SU(2)^{3}\times U(1)^{4})$\\
\hline
$E_{7}/(SU(3)\times SU(2)\times U(1)^{4})$&
$E_{7}/(SU(5)\times U(1)^{3})$\\
\hline
$E_{7}/(SU(4)\times SU(2)\times U(1)^{3})$&
$E_{8}/(SU(2)\times U(1)^{6})$\\
\hline
$E_{8}/(SU(3)\times U(1)^{5})$&
$E_{8}/(SU(3)\times SU(2)\times U(1)^{4})$\\
\hline
$E_{8}/(SU(5)\times U(1)^{3})$&
$E_{8}/(SU(2)^{2}\times U(1)^{5})$\\
\hline
$E_{8}/(SU(4)\times U(1)^{4})$&
$E_{8}/(SU(2)^{3}\times U(1)^{4})$\\
\hline
$E_{8}/(SU(4)\times SU(2)\times U(1)^{3})$&
$F_{4}/(SU(2)\times U(1)^{3})$\\
\hline
$SO(2n+4)/U(n)\times U(1)^{2}$&\\
\hline
\end{tabular}
\end{table}

\section{Chern numbers on generalized flag manifolds}

\subsection{Generalized flag manifolds of the classical Lie group $SU(n)$} \label{sec1}
\subsubsection{The generalized flag manifolds with 3 isotropy summands}
Recall the notation about generalized flag manifold
$$
\mathbb{F}(n;n_1,n_2,n_3)=SU(n)/S(U(n_1)\times U(n_2)\times U(n_3)),
$$
where $n=n_1+n_2+n_3$.

The isotropy representation  of $\mathbb{F}(n;n_1,n_2,n_3)$ splits into 3 isotropy summands. Therefore $\mathbb{F}(n;n_1,n_2,n_3)$ admits 4 invariant almost complex structures, up to conjugation:
$$
\begin{array}{ccccccc}
J_{1}=(+,+,+)&&&&&&J_{3}=(+,+,-)\\
J_{2}=(-,+,+)&&&&&&J_{4}=(+,-,+).
\end{array}
$$

In the next proposition we obtain a complete description of almost complex structures on several generalized flag manifolds. This is done computing the Chern number $c_1^n$ where $n$ is the complex dimension of the manifold.

\begin{proposition}\label{complete3}
The following generalized flag manifolds 

\begin{table}[h] 

\centering


\begin{tabular}{|c|c|c|} 
\hline 

$M$ & $\dim$ (real) & $\chi(M)$\\ 

\hline
\hline
$\mathbb{F}(6;1,2,3)$ & $22$ & $60$\\
$\mathbb{F}(7;1,2,4)$ & $28$ & $105$\\
$\mathbb{F}(8;1,2,5)$ & $34$ & $168$\\
$\mathbb{F}(8;1,3,4)$ & $38$ & $280$\\
\hline
\end{tabular}
\label{tab2-novo}
\end{table}
have precisely four invariant almost complex structures up to conjugation and equivalence; three of them are integrable and the one is non-integrable.
\end{proposition}
\begin{proof}
We will proof the proposition just for $\mathbb{F}(7;1,2,4)$. To the other manifolds, the proof are similar.
The Cartan sub-algebra $\mathfrak{h}$ of $\mathfrak{su}(n)$ is given by
$$
\got{h}=\{\diag(x_{1},\cdots,x_{7}):x_{1}+\cdots+x_{7}=0\}.
$$

According lemma \ref{order}, the invariant almost complex structures  $J_{1}$, $J_{2}$ and $J_{3}$ are integrable and  $J_{4}$ is not integrable. In fact, by analyzing irreducible summands of the isotropy representation and the Weyl chamber of $\mathfrak{h}$ one can order the coordinates of the Cartan sub-algebra of  $(\mathbb{F}(7;1,2,4),J_{1})$, $(\mathbb{F}(7;1,2,4),J_{2})$ and $(\mathbb{F}(7;1,2,4),J_{3})$ in the following way:

$J_{1}$: $x_{1}>x_{2}>x_{3}>x_{4}>x_{5}>x_{6}>x_{7}$;

$J_{2}$: $x_{2}>x_{3}>x_{1}>x_{4}>x_{5}>x_{6}>x_{7}$;

$J_{3}$: $x_{1}>x_{4}>x_{5}>x_{6}>x_{7}>x_{2}>x_{3}$.

On the other hand there is no ordering on the coordinates of $(\mathbb{F}(7;1,2,4),J_{4})$. In fact, the manifold $(\mathbb{F}(7;1,2,4),J_{4})$ have the following roots: 

$x_{1}-x_{2}>0$, $x_{1}-x_{3}>0$, $x_{1}-x_{4}<0$, $x_{1}-x_{5}<0$, $x_{1}-x_{6}<0$, $x_{1}-x_{7}<0$, $x_{2}-x_{4}>0$ and $x_{2}-x_{5}>0$.

In this way, 
$x_{1}>x_{2}$, $x_{2}>x_{4}$, but $x_{4}>x_{1}$, that is, $x_{1}>x_{2}>x_{4}>x_{1}$, and this is a contradiction.

The cohomology ring of  $\mathbb{F}(7;1,2,4)$ is given by:
$$
H^{*}(\mathbb{F}(7;1,2,4),\mathbb{R})=\dfrac{\mathbb{R}[x_{1},x_{2},x_{3},x_{4},x_{5},x_{6},x_{7}]}{\langle e_{1},e_{2},e_{3},e_{4},e_{5},e_{6},e_{7}\rangle},
$$ 
where $e_{k}=x_{1}^{k}+x_{2}^{k}+x_{3}^{k}+x_{4}^{k}+x_{5}^{k}+x_{6}^{k}+x_{7}^{k}$.

The class $x_{2}x_{3}x_{4}^{3}x_{5}^{3}x_{6}^{3}x_{7}^{3}\in H^{28}(\mathbb{F}(7;1,2,4),\mathbb{R})$ generates the top dimensional cohomology class, where  $c_{14}=105x_{2}x_{3}x_{4}^{3}x_{5}^{3}x_{6}^{3}x_{7}^{3}$.

Computing the Chern number $c_{1}^{14}$ to the four invariant almost complex structures, we have

\begin{table}[h] 

\centering

\caption{ } 

\begin{tabular}{|c|c|c|c|c|} 

\hline 

&$J_1=(+,+,+)$ & $J_2=(-,+,+)$ & $J_3=(+,+,-)$ & $J_4=(-,+,-)$\\ 
\hline
\hline
$c_{1}^{14}$& 4169710642825728 & 3967580897280000 & 5340215200320000 & 68881612800 \\
\hline
\end{tabular}
\label{tab-dif}
\end{table}

Since the Chern numbers computed in the table \ref{tab-dif} are distinct for the 4 invariant complex structures we conclude that these structures are not equivalent.


\end{proof}

\begin{remark}
We list in the next table the relevant Chern numbers used in the proof of the proposition \ref{complete3}.

\footnotesize
\begin{table}[h] 

\centering


\begin{tabular}{|c|c|c|c|c|c|} 

\hline 

$M$& $c_1^n$ &$J_1=(+,+,+)$ & $J_2=(-,+,+)$ & $J_3=(+,+,-)$ & $J_4=(-,+,-)$\\ 
\hline
\hline
$\mathbb{F}(6;1,2,3)$& $c_{1}^{11}$& -166320000000 & 187110000000 & -156539053440 & 0 \\
$\mathbb{F}(8;1,2,5)$& $c_{1}^{17}$& -207657272688465600000& -199318721129508524544 & 303212843288789930496& 1250749500000000\\
 $\mathbb{F}(8;1,3,4)$ & $c_{1}^{19}$ & 301923064586776419730944&262989979268101525440000 & 347992057571330652480000&363738375000000000\\
\hline
\end{tabular}
\label{tab2}
\end{table}

\normalfont

\end{remark}

\begin{proposition}
The 16-dimensional generalized flag manifold $\mathbb{F}(5;1,2,2)$ has 3 invariant almost complex structures, up to conjugation and equivalence: two of them are integrable and non-equivalent and third is non-integrable.
\end{proposition}

\begin{proof}
We first remark that the almost complex structures $J_{1}$, $J_{2}$ and $J_{3}$ are integrable and $J_{4}$ is non-integrable. In fact, one can ordering the coordinates of the Cartan sub-algebra of $(M_{1},J_{1})$, $(M_{1},J_{2})$ and $(M_{1},J_{3})$ in the following way: 
$$
\begin{array}{c}
J_{1}: \ x_{1}>x_{2}>x_{3}>x_{4}>x_{5},\\
J_{2}: \ x_{2}>x_{3}>x_{1}>x_{4}>x_{5},\\
J_{3}: \ x_{1}>x_{4}>x_{5}>x_{2}>x_{3}.
\end{array}
$$
By Lemma  \ref{order}, we conclude that  $J_{1}$, $J_{2}$ and $J_{3}$ are integrable. On the other hand, the almost complex structure $J_{4}$ is not integrable since there is not exist an ordering compatible with $J_4$ on the coordinates  $x_{1},x_{2},x_{3},x_{4},x_{5}$  such that the roots are positive.  

Now we will prove that the complex structures $J_{1}$ and $J_{3}$ are equivalent. Consider the following map:
$$
\begin{array}{rcl}
f: \got{h}\subset A_{4}&\rightarrow&\got{h}\subset A_{4}\\
\diag(x_{1},x_{2},x_{3},x_{4},x_{5})&\mapsto& \diag(x_{1},x_{4},x_{5},x_{2},x_{3}).
\end{array}
$$
The map $f$ is an automorphism of $\got{h}$ that sends the root system of  $J_{1}$ into root system of  $J_{3}$ and fixes the root system of  $S(U(1)\times U(2)\times U(2))$.  By Proposition \ref{si} follow that $J_{1}$ and $J_{3}$ are equivalent.

The cohomology ring of $\mathbb{F}(5;1,2,2)$ is given by 
$$
H^{*}(\mathbb{F}(5;1,2,2),\mathbb{R})=\dfrac{\mathbb{R}[x_{1},x_{2},x_{3},x_{4},x_{5}]}{\langle e_{1},e_{2},e_{3},e_{4},e_{5}\rangle},
$$ 
where $e_{k}=x_{1}^{k}+x_{2}^{k}+x_{3}^{k}+x_{4}^{k}+x_{5}^{k}$.

We can conclude the proof analyzing the Chern numbers of  $\mathbb{F}(5;1,2,2)$ listed in the Table \ref{tab3}.

\end{proof}

\begin{table}[h] 

\centering

\caption{Chern numbers of $\mathbb{F}(5;1,2,2)$} 

\begin{tabular}{|c|c|c|c|c|} 

\hline 

&$J_1=(+,+,+)$ & $J_2=(-,+,+)$ & $J_3=(-,+,-)$ & $J_4=(+,+,-)$\\ 
\hline
\hline
$c_{8}$& 30 & 30 & 30 & 30\\
\hline
$c_{1}^{8}$& 15805440 & 14696640 &2240 & 15805440 \\
\hline
$c_{1}^{6}c_{2}$ & 7579680 & 7085880 & 760 & 7579680\\ 
\hline
$c_{1}^{5}c_{3}$ & 2262960 & 2143260 & 220 & 2262960\\
\hline
$c_{1}^{4}c_{4}$ & 459990 & 444690 & 210 & 459990\\
\hline
$c_{1}^{4}c_{2}^{2}$ & 3637010 & 3419010 & 290 & 3637010\\
\hline
$c_{1}^{3}c_{5}$ & 66510 & 65610 & 90 & 66510\\
\hline
$c_{1}^{3}c_{2}c_{3}$ & 1087270 & 1035720 & 180 & 1087270\\
\hline
$c_{1}^{2}c_{6}$ & 7020 & 7020 & 60 & 7020\\
\hline
$c_{1}^{2}c_{2}^{3}$ & 1746170 & 1650870 & 110 & 1746170\\
\hline 
$c_{1}^{2}c_{3}^{2}$ & 325940 & 314640 & 360 & 325940\\
\hline
$c_{1}^{2}c_{2}c_{4}$ & 221430 & 215280 & 140 & 221430\\
\hline
$c_{1}c_{7}$ & 540 & 540 & 60 & 540\\
\hline
$c_{1}c_{2}^{2}c_{3}$ & 522690 & 500790 & 70 & 522690\\
\hline
$c_{1}c_{2}c_{5}$ & 32070 & 31770 & 10 & 32070\\
\hline
$c_{1}c_{3}c_{4}$ & 66660 & 65610 & 230 & 66660\\
\hline
$c_{2}^{4}$ & 838840 & 797640 & 40 & 838840\\
\hline
$c_{2}^{2}c_{4}$ & 106660 & 104260 & 60 & 106660\\
\hline
$c_{2}c_{6}$ & 3390 & 3390 & 10 & 3390\\
\hline
$c_{2}c_{3}^{2}$ & 156880 & 152280 & 40 & 156880\\
\hline
$c_{3}c_{5}$ & 9690 & 9690 & -30 & 9690\\
\hline
$c_{4}^{2}$ & 13730 & 13730 & 130 & 13730\\
\hline
\end{tabular}
\label{tab3}
\end{table}

\begin{remark}
We now apply the Hirzebruch-Riemann-Roch Theorem. Since the arithmetic genus coincides with the Todd genus we have 
$$
\begin{array}{ccl}
\sum_{q=0}^{8}(-1)^{q}h^{0,q}&=&(1/3628800)(-3c_{8}-3c_{1}^{8}+24c_{1}^{6}c_{2}-50c_{1}^{4}c_{2}^{2}+8c_{1}^{2}c_{2}^3\\
&+&21c_{2}^{4}-14c_{1}^{5}c_{3}+26c_{1}^{3}c_{2}c_{3}+50c_{1}c_{2}^{2}c_{3}+
3c_{1}^{2}c_{3}^{2}\\&-&8c_{2}c_{3}^{2}+14c_{1}^{4}c_{4}-19c_{1}^{2}c_{2}c_{4}-
34c_{2}^{2}c_{4}-13c_{1}c_{3}c_{4}+5c_{4}^{2}\\&-&7c_{1}^3c_{5}-16c_{1}c_{2}c_{5}
+3c_{3}c_{5}+7c_{1}^{2}c_{6}+13c_{2}c_{6}+3c_{1}c_{7}).
\end{array}
$$

If $M$ is a smooth manifold with a complex structure $J$ we have (cf. \cite{opa}, \cite{princ})
$$\sum_{q=0}^{8}(-1)^{q}h^{0,q}=1.$$ 
Therefore
$$
\begin{array}{ccl}
3628800&=&-3c_{8}-3c_{1}^{8}+24c_{1}^{6}c_{2}-50c_{1}^{4}c_{2}^{2}+8c_{1}^{2}c_{2}^3\\
&+&21c_{2}^{4}-14c_{1}^{5}c_{3}+26c_{1}^{3}c_{2}c_{3}+50c_{1}c_{2}^{2}c_{3}+
3c_{1}^{2}c_{3}^{2}\\&-&8c_{2}c_{3}^{2}+14c_{1}^{4}c_{4}-19c_{1}^{2}c_{2}c_{4}-
34c_{2}^{2}c_{4}-13c_{1}c_{3}c_{4}+5c_{4}^{2}\\&-&7c_{1}^3c_{5}-16c_{1}c_{2}c_{5}
+3c_{3}c_{5}+7c_{1}^{2}c_{6}+13c_{2}c_{6}+3c_{1}c_{7}.
\end{array}
$$
\end{remark}

The next theorem classify the almost complex structures in the family  $\mathbb{F}(3n;n,n,n)=SU(3n)/S(U(n)\times U(n)\times U(n))$ .

\begin{theorem}\label{flag6222}
The family of generalized flag manifolds  $\mathbb{F}(3n;n,n,n)=SU(3n)/S(U(n)\times U(n)\times U(n))$ has two invariant almost complex structures, up to conjugation and equivalence: one is an integrable structure and the other is non-integrable.
\end{theorem}
\begin{proof}
The Cartan sub-algebra  $\got{h}$ of $\got{su}(n)$ is given by 
$$
\got{h}=\{\diag(x_{1},\cdots,x_{3n}):x_{1}+\cdots+x_{3n}=0\}.
$$

The isotropy representation  $\mathbb{F}(3n;n,n,n)$ decomposes into three isotropic summands:
$$
\begin{array}{l}
\got{m}_{1}=\got{u}_{1,n+1}\oplus\cdots\oplus \got{u}_{1,2n}\cdots \got{u}_{n,n+1}\oplus\cdots\oplus \got{u}_{n,2n},\\
\got{m}_{2}=\got{u}_{1,2n+1}\oplus\cdots\oplus \got{u}_{1,3n}\cdots \got{u}_{n,2n+1}\oplus\cdots\oplus \got{u}_{n,3n},\\
\got{m}_{3}=\got{u}_{n+1,2n+1}\oplus\cdots\oplus \got{u}_{n+1,3n}\cdots \got{u}_{2n,2n+1}\oplus\cdots\oplus \got{u}_{2n,3n},
\end{array}
$$
with $\got{u}_{ij}=\got{su}(3n)\cap (\got{g}_{ij}\oplus \got{g}_{ji})$, where $\mathfrak{g}_{ij}$ is the (complex) root space associated to the root $\alpha_{ij}=x_i-x_j$ of $\mathfrak{sl}(3n.\mathbb{C})=(\mathfrak{su}(3n))^{\mathbb{C}}$. 

According to proposition \ref{numest}, $\mathbb{F}(3n;n,n,n)$ admits four invariant almost complex structures, up to conjugation:
$$
\begin{array}{cccccccccc}
J_{1}=(+,+,+)&&&&&&&&&J_{2}=(-,+,+)\\
J_{3}=(+,+,-)&&&&&&&&&J_{4}=(+,-,+).
\end{array}
$$
By lemma \ref{order}, the structures  $J_{1}$, $J_{2}$ and $J_{3}$ are integrable. In fact, by analyzing irreducible summands of the isotropy representation and the Weyl chamber of $\mathfrak{h}$ we consider the following order for the coordinates on the Cartan sub-algebra:
$$
\begin{array}{l}
J_{1}:\ \ x_{1}>\cdots >x_{n}>x_{n+1}>\cdots>x_{2n}>x_{2n+1}>\cdots>x_{3n},\\
J_{2}:\ \ x_{n+1}>\cdots>x_{2n}>x_{1}>\cdots >x_{n}>x_{2n+1}>\cdots>x_{3n},\\
J_{3}:\ \ x_{1}>\cdots >x_{n}>x_{2n+1}>\cdots>x_{3n}>x_{n+1}>\cdots>x_{2n}.
\end{array}
$$
But there no exist an ordering to coordinates on the Cartan sub-algebra of $(\mathbb{F}(3n;n,n,n),J_{4})$ where the roots are positive, hence $J_4$ is not integrable. We now proof:
\\

{\bf Claim}: $J_{1}$, $J_{2}$ and $J_3$ are equivalent.

We being proving that $J_1$ and $J_2$ are equivalent. We consider the map  $g:\got{h}\longrightarrow \got{h}$ given by $$g(x_{1},\cdots,x_{n},x_{n+1}\cdots,x_{2n},x_{2n+1},\cdots x_{3n})=
(x_{n+1},\cdots,x_{2n},x_{1},\cdots ,x_{n},x_{2n+1},\cdots,x_{3n}).$$  

The map  $g$ is an automorphism of  $\got{h}$ that sends the root system of $J_{1}$ onto the root system of  $J_{2}$ and keep fixed the root system of $S(U(n)\times U(n)\times U(n))$.

In the same way, one can prove that  $J_{1}$ and $J_{3}$ are equivalent: just consider the following automorphism $f$ of $\mathbb{h}$ given by $$f(x_{1},\cdots,x_{n},x_{n+1}\cdots,x_{2n},x_{2n+1},\cdots x_{3n})=
(x_{1},\cdots ,x_{n},x_{2n+1},\cdots,x_{3n},x_{n+1},\cdots,x_{2n}),$$this way we conclude the proof.
\end{proof}

\begin{remark}
The invariant complex structure on $\mathbb{F}(3n;n,n,n)$ described in the proposition \ref{flag6222} is compatible with the invariant Kahler-Einstein metric described by Arvanitoyeorgos in \cite{grego1}.
\end{remark}

\subsubsection{The full flag manifold $\mathbb{F}(4)=SU(4)/T$}

Consider the 12-dimensional full flag manifold $\mathbb{F}(4)=SU(4)/T$ with Euler characteristic $24$.

According to \cite{livro} a Cartan sub-algebra $\got{h}$ of $\got{su}(4)$ is given by
$$
\got{h}=\{\diag(x,y,z,w):x+y+z+w=0,\ x,y,z,w\in\mathbb{C}\}.
$$

The isotropy representation of $\mathbb{F}(4)$ admits six isotropic summands, that is,
$$
\begin{array}{ccl}
\got{m}&=&\got{m}_{1}\oplus\got{m}_{2}\oplus \got{m}_{3}\oplus\got{m}_{4}\oplus\got{m}_{5}\oplus \got{m}_{6}\\
&=&\got{u}_{12}\oplus \got{u}_{13}\oplus \got{u}_{14}\oplus \got{u}_{23}\oplus \got{u}_{24}\oplus \got{u}_{34},
\end{array}
$$
where $\got{u}_{ij}=\got{su}(4)\cap (\got{g}_{ij}\oplus \got{g}_{ji})$, and $\got{g}_{ij}$ is a complex root space associated to the root $\alpha_{ij}$.

By proposition \ref{numest} the flag manifold $\mathbb{F}(4)$ admits $2^{5}=32$ invariant almost complex structures up to conjugation. According to \cite{parede}, $\mathbb{F}(4)$ admits exactly $4$ invariant almost complex structures, up to conjugation and equivalence:
$$
\begin{array}{cccccc}
J=(+,+,+,+,+,+)&&&&&I_{1}=(+,+,+,-,+,-)\\
I_{2}=(+,+,-,+,+,+)&&&&&I_{3}=(-,+,+,-,+,+),\\
\end{array}
$$
where $J$ is integrable and $I_{1}$, $I_{2}$ and $I_{3}$ are not integrable.

The cohomology ring of $\mathbb{F}(4)$ is given by 
$$
H^{*}(\mathbb{F}(4),\mathbb{R})=\dfrac{\mathbb{R}[x,y,z,w]}{\langle e_{1},e_{2},e_{3},e_{4}\rangle},
$$
where $e_{k}=x^{k}+y^{k}+z^{k}+w^{k}$, with $x,y,z,w$ in degree $2$ and $k=1,2,3,4$.

The Chern numbers of $\mathbb{F}(4)$ are in the table \ref{tab5}.  

\begin{table}[h] 

\centering

\caption{Chern numbers of $\mathbb{F}(4)$} 

\begin{tabular}{|c|c|c|c|} 

\hline 

$J$ & $I_{1}$ & $I_{2}$ & $I_{3}$\\ 

\hline
\hline


$c_{1}^{6}=46080$&$c_{1}^{6}=0$&$c_{1}^{6}=0$&$c_{1}^{6}=0$\\
\hline
$c_{1}^{4}c_{2}=23040$&$c_{1}^{4}c_{2}=0$&$c_{1}^{4}c_{2}=0$&$c_{1}^{4}c_{2}=0$ \\
\hline
$c_{1}^{2}c_{2}^{2}=11520$&$c_{1}^{2}c_{2}^{2}=0$&$c_{1}^{2}c_{2}^{2}=0$&$c_{1}^{2}c_{2}^{2}=0$\\ 
\hline
$c_{1}^{3}c_{3}=7360$&$c_{1}^{3}c_{3}=384$&$c_{1}^{3}c_{3}=-64$&$c_{1}^{3}c_{3}=384$\\
\hline
$c_{2}^{3}=5760$&$c_{2}^{3}=0$&$c_{2}^{3}=0$&$c_{2}^{3}=0$\\
\hline
$c_{1}c_{2}c_{3}=3680$&$c_{1}c_{2}c_{3}=192$&$c_{1}c_{2}c_{3}=-32$&$c_{1}c_{2}c_{3}=192$\\
\hline 
$c_{1}^{2}c_{4}=1600$&$c_{1}^{2}c_{4}=384$&$c_{1}^{2}c_{4}=-64$&$c_{1}^{2}c_{4}=384$\\
\hline
$c_{3}^{2}=1168$&$c_{3}^{2}=144$&$c_{3}^{2}=-16$&$c_{3}^{2}=144$\\
\hline
$c_{2}c_{4}=800$&$c_{2}c_{4}=192$&$c_{2}c_{4}=-32$&$c_{2}c_{4}=192$\\
\hline
$c_{1}c_{5}=240$&$c_{1}c_{5}=144$&$c_{1}c_{5}=-96$&$c_{1}c_{5}=144$\\
\hline
$c_{6}=24$&$c_{6}=24$&$c_{6}=-24$&$c_{6}=24$\\
\hline
\end{tabular}
\label{tab5}
\end{table}

The Hirzebruch-Riemann-Roch Theorem provides the following relations between the Chern numbers on the flag manifold $(\mathbb{F}(4),J)$, where $J$ is invariant complex structure,
$$
\begin{array}{ccl}
60480&=&2c_{6}-2c_{1}c_{5}-9c_{2}c_{4}-5c_{1}^{2}c_{4}-c_{3}^{2}+11c_{1}c_{2}c_{3}+5c_{1}^{3}c_{3}\\
&+&10c_{2}^{3}+11c_{1}^{2}c_{2}^{2}-12c_{1}^{4}c_{2}+2c_{1}^{6}.
\end{array}
$$

\subsubsection{The full flag manifold $\mathbb{F}(5)=SU(5)/T$}

Consider the 20-dimensional full flag manifold $\mathbb{F}(5)=SU(5)/T$ with Euler characteristic $120$.

According \cite{livro} a Cartan sub-algebra $\got{h}$ of $\got{su}(5)$ is given by 
$$
\got{h}=\{\diag(x,y,z,w,u):x+y+z+w+u=0,\ x,y,z,w,u\in\mathbb{C}\}.
$$

The isotropic representation of $\mathbb{F}(5)$ admits $10$ isotropic summands, that is,
$$
\begin{array}{ccl}
\got{m}&=&\got{m}_{1}\oplus\got{m}_{2}\oplus \got{m}_{3}\oplus\got{m}_{4}\oplus\got{m}_{5}\oplus \got{m}_{6}\oplus \got{m}_{7}\oplus\got{m}_{8}\oplus\got{m}_{9}\oplus \got{m}_{10}\\
&=&\got{u}_{12}\oplus \got{u}_{13}\oplus \got{u}_{14}\oplus \got{u}_{15}\oplus \got{u}_{23}\oplus \got{u}_{24}\oplus \got{u}_{25}\oplus \got{u}_{34}\oplus \got{u}_{35}\oplus \got{u}_{45},
\end{array}
$$
where $\got{u}_{ij}=\got{su}(5)\cap (\got{g}_{ij}\oplus \got{g}_{ji})$, and $\got{g}_{ij}$ is a complex root space associated to the root $\alpha_{ij}$.

By proposition \ref{numest}, $\mathbb{F}(5)$ admits $2^{9}=512$ invariant almost complex structures, up to conjugations. According \cite{parede}, $\mathbb{F}(5)$ admits precisely $12$ invariant almost complex structures, up to conjugation and equivalence:
$$
\begin{array}{cccc}
J=(+,+,+,+,+,+,+,+,+,+)&&&I_{1}=(-,+,+,+,-,+,+,+,+,+)\\
I_{2}=(-,+,+,+,-,+,+,-,+,+)&&&I_{3}=(+,+,+,+,-,+,+,-,+,+)\\
I_{4}=(+,+,+,+,+,+,+,-,+,-)&&&I_{5}=(+,+,+,+,-,+,+,-,+,-)\\
I_{6}=(+,+,+,-,+,+,+,+,+,+)&&&I_{7}=(+,+,+,-,+,+,+,-,+,-)\\
I_{8}=(+,+,+,-,+,+,-,+,+,+)&&&I_{9}=(-,+,+,-,-,+,-,+,+,+)\\
I_{10}=(+,+,+,-,-,+,+,-,+,+)&&&I_{11}=(-,-,+,+,-,-,+,-,-,-),\\
\end{array}
$$
where $J$ is the canonical invariant almost complex structure (integrable).

The cohomology ring of $\mathbb{F}(5)$ is given by 
$$
H^{*}(\mathbb{F}(5),\mathbb{R})=\dfrac{\mathbb{R}[x,y,z,w,u]}{\langle e_{1},e_{2},e_{3},e_{4},e_{5}\rangle},
$$
where $e_{k}=x^{k}+y^{k}+z^{k}+w^{k}+u^{k}$ with $x,y,z,w,u$ in degree $2$ and $k=1,2,3,4,5$.  .

The Chern numbers of $\mathbb{F}(5)$ with the respective invariant almost complex structure are in the tables \ref{tabf51} and \ref{tabf54}.

\begin{table}[h] 

\centering

\caption{Chern numbers of $\mathbb{F}(5)$} 

\begin{tabular}{|c|c|c|c|c|c|c|} 

\hline 

 & $J$ & $I_{1}$& $I_{2}$ & $I_{3}$ &  $I_{4}$ &  $I_{5}$ \\ 

\hline
\hline
$c_{1}^{10}$&3715891200&0& 0 & 0 & 0 & 0 \\
\hline
$c_{1}^{8}c_{2}$&1857945600&0& 0 & 0& 0 & 0\\
\hline
$c_{1}^{7}c_{3}$&610086400&5806080& -1236480 & 6881280& -1236480 & 5806080\\
\hline
$c_{1}^{6}c_{4}$&14560000&5806080& -1236480 & 6881280& -1236480 & 5806080\\
\hline
$c_{1}^{5}c_{5}$&26464000&2753280& -629760 & 3194880& -629760 & 2753280 \\
\hline
$c_{1}^{4}c_{6}$&3744000&804480& -219520 & 890880& -219520 & 804480\\
\hline 
$c_{1}^{3}c_{7}$&415200&157920& -56320 & 165120& -56320 & 157920 \\
\hline
$c_{1}^{2}c_{8}$&36000&21120& -10720 & 21120 & -10720 & 21120 \\
\hline
$c_{1}^{6}c_{2}^{2}$&928972800&0& 0 & 0 & 0 & 0 \\
\hline
$c_{1}^{4}c_{2}^{3}$&464486400&0& 0 & 0& 0 & 0  \\
\hline
$c_{1}^{2}c_{2}^{4}$&232243200&0& 0 & 0 & 0 & 0 \\
\hline
$c_{2}^{5}$&116121600&0& 0 & 0 & 0 & 0 \\
\hline
$c_{1}^{4}c_{3}^{2}$&100160000&1908480& -416000 & 2273280 & -416000 & 1908480\\
\hline
$c_{1}c_{3}^{3}$&16442400&470880& -104640 & 564480 & -104640 & 470880 \\
\hline
$c_{1}^{2}c_{2}c_{3}^{2}$&50080000&954240& -208000 & 1136640 & -208000 & 954240 \\
\hline
$c_{1}^{3}c_{2}^{2}c_{3}$&152521600&1451520& -309120 & 1720320 & -309120 & 1451520 \\
\hline
$c_{1}c_{2}^{3}c_{3}$&76260800&725760& -154560 & 860160& -154560 & 725760 \\
\hline
$c_{2}^{2}c_{3}^{2}$&2504000&477120& -104000 & 568320& -104000 & 477120  \\
\hline
$c_{1}^{2}c_{4}^{2}$&5699200&456960& -106880 & 552960& -106880 & 456960\\
\hline
$c_{2}c_{4}^{2}$&2849600&228480& -53440 & 276480 & -53440 & 228480  \\
\hline
$c_{1}^{3}c_{3}c_{4}$&23899200&1182720& -261440 & 1413120& -261440 & 1182720 \\
\hline
$c_{1}c_{9}$&2400&1920& -1440 & 1920& -1440 & 1920 \\
\hline
$c_{1}^{5}c_{2}c_{3}$ &305043200&2903040& -618240 & 3440640& -618240 & 2903040\\
\hline
$c_{1}^{3}c_{2}c_{5}$&13232000&1376640& -314880 & 1597440 & -314880 & 1376640\\
\hline
$c_{1}^{2}c_{2}^{2}c_{4}$&36400000&1451520& -309120 & 1720320& -309120 & 1451520\\
\hline
$c_{1}^{2}c_{2}c_{6}$&1872000&402240& -109760 & 445440& -109760 & 402240\\
\hline
$c_{1}^{2}c_{3}c_{5}$&43424400&493920& -117600 & 577920& -117600 & 493920\\
\hline 
$c_{1}c_{2}c_{7}$&207600&78960& -28160 & 82560& -28160 & 78960\\
\hline
$c_{1}^{4}c_{2}c_{4}$&72800000&2903040& -618240 & 3440640& -618240 & 2903040\\
\hline
$c_{1}c_{2}^{2}c_{5}$&661600&688320& -157440 & 798720& -157440 & 688320\\
\hline
$c_{1}c_{2}c_{3}c_{4}$&11949600&591360& -130720 & 706560& -130720 & 591360\\
\hline
$c_{1}c_{3}c_{6}$&614000&138000& -39200 & 153600 & -39200 & 138000\\
\hline
$c_{1}c_{4}c_{5}$&1034400&149760& -38880 & 178560& -38880 & 149760\\
\hline
$c_{2}c_{8}$&18000&10560& -5360 & 10560& -5360 & 10560\\
\hline
$c_{2}^{3}c_{4}$&18200000&7257760& -154560 & 860160& -154560 & 725760\\
\hline
$c_{2}^{2}c_{6}$&936000&201120& -54880 & 222720 & -54880 & 201120\\
\hline
$c_{2}c_{3}c_{5}$&2171200&246960& -58800 & 288960& -58800 & 246960\\
\hline
$c_{3}c_{7}$&68040&26520& -9880 & 27720& -9880 & 26520\\
\hline
$c_{3}^{2}c_{4}$&3922400&232320& -52640 & 280320& -52640 & 232320\\
\hline
$c_{4}c_{6}$&146000&37440& -11760 & 42240& -11760 & 37440\\
\hline
$c_{5}^{2}$&187360&39360& -11840 & 46560& -11840 & 39360\\
\hline
$c_{10}$&120&120& -120 & 120& -120 & 120\\
\hline

\end{tabular}
\label{tabf51}
\end{table}

\begin{table}[h] 

\centering

\caption{Chern numbers of $\mathbb{F}(5)$ - continuation} 

\begin{tabular}{|c|c|c|c|c|c|c|} 

\hline 

 & $I_{6}$ & $I_{7}$ & $I_{8}$ & $I_{9}$ & $I_{10}$& $I_{11}$ \\ 

\hline
\hline

$c_{1}^{10}$& 0 & 0& 0 & 0& 0 & 0 \\
\hline
$c_{1}^{8}c_{2}$& 0 & 0& 0 & 0& 0 & 0\\
\hline
$c_{1}^{7}c_{3}$& -573440 & 0& 53760 & -17920 & 53760 & -17920 \\
\hline
$c_{1}^{6}c_{4}$& -573440 & 0 & 53760 & -17920 & 53760 & -17920 \\
\hline
$c_{1}^{5}c_{5}$& -335360 & 0& 19200 & -6400 & 19200 & -6400 \\
\hline
$c_{1}^{4}c_{6}$& -129280 & -10240& 2560 & 1280 & 2560 & 1280 \\
\hline 
$c_{1}^{3}c_{7}$& -34240 & -10240 & 1120 & 1760& 1120 & 1760 \\
\hline
$c_{1}^{2}c_{8}$& -6400 & -4480 & 1120 & 800& 1120 & 800 \\
\hline
$c_{1}^{6}c_{2}^{2}$& 0 & 0 & 0 & 0& 0 & 0  \\
\hline
$c_{1}^{4}c_{2}^{3}$& 0 & 0 & 0 & 0& 0 & 0  \\
\hline
$c_{1}^{2}c_{2}^{4}$& 0 & 0 & 0 & 0 & 0 & 0 \\
\hline
$c_{2}^{5}$& 0 & 0 & 0 & 0 & 0 & 0 \\
\hline
$c_{1}^{4}c_{3}^{2}$& -161280 & -20480& 20480 & -2560 & 20480 & -2560 \\
\hline
$c_{1}c_{3}^{3}$& -33600 & -9600& 5280 & 1440 & 5280 & 1440 \\
\hline
$c_{1}^{2}c_{2}c_{3}^{2}$& -80640 & -10240& 10240 & -1280 & 10240 & -1280  \\
\hline
$c_{1}^{3}c_{2}^{2}c_{3}$& -143360 & 0& 13440 & -4480& 13440 & -4480 \\
\hline
$c_{1}c_{2}^{3}c_{3}$& -71680 & 0 & 6720 & -2240 & 6720 & -2240\\
\hline
$c_{2}^{2}c_{3}^{2}$& -40320 & -5120& 5120 & -640& 5120 & -640  \\
\hline
$c_{1}^{2}c_{4}^{2}$& -17920 & -20480 & 7040 & 1920& 7040 & 1920  \\
\hline
$c_{2}c_{4}^{2}$& -8960 & -10240 & 3520 & 960& 3520 & 960 \\
\hline
$c_{1}^{3}c_{3}c_{4}$& -89600 & -20480 & 13760 & -320& 13760 & -320 \\
\hline
$c_{1}c_{9}$& -960 & -960& 480 & 480& 480 & 480\\
\hline
$c_{1}^{5}c_{2}c_{3}$ & -286720 & 0& 26880 & -8960& 26880 & -8960\\
\hline
$c_{1}^{3}c_{2}c_{5}$& -167680 & 0& 9600 & -3200& 9600 & -3200\\
\hline
$c_{1}^{2}c_{2}^{2}c_{4}$& -143360 & 0& 13440 & -4480& 13440 & -4480\\
\hline
$c_{1}^{2}c_{2}c_{6}$& -64640 & -5120& 1280 & 640& 1280 & 640\\
\hline
$c_{1}^{2}c_{3}c_{5}$& -45760 & -9600& 4800 & 640& 4800 & 640\\
\hline 
$c_{1}c_{2}c_{7}$& -17120 & -5120& 560 & 880& 560 & 880\\
\hline
$c_{1}^{4}c_{2}c_{4}$& -286720 & 0& 26880 & -8960& 26880 & -8960\\
\hline
$c_{1}c_{2}^{2}c_{5}$& -83840 & 0& 4800 & -1600& 4800 & -1600\\
\hline
$c_{1}c_{2}c_{3}c_{4}$& -44800 & -10240& 6880 & -160& 6880 & -160\\
\hline
$c_{1}c_{3}c_{6}$& -17760 & -4160& 560 & 1520& 560 & 1520\\
\hline
$c_{1}c_{4}c_{5}$& -3840 & -9600& 2400 & 1440& 2400 & 1440\\
\hline
$c_{2}c_{8}$& -3200 & -2240& 560 & 400& 560 & 400\\
\hline
$c_{2}^{3}c_{4}$& -71680 & 0& 6720 & -2240& 6720 & -2240\\
\hline
$c_{2}^{2}c_{6}$& -32320 & -2560& 640 & 320& 640 & 320\\
\hline
$c_{2}c_{3}c_{5}$& -22880 & -4800& 2400 & 320& 2400 & 320\\
\hline
$c_{3}c_{7}$& -4840 & -2120& 40 & 840& 40 & 840\\
\hline
$c_{3}^{2}c_{4}$& -13440 & -7040& 2720 & 1760& 2720 & 1760\\
\hline
$c_{4}c_{6}$& -1600 & -2880& 240 & 1360& 240 & 1360\\
\hline
$c_{5}^{2}$& -1760 & -4320& 800 & 1120& 800 & 1120\\
\hline
$c_{10}$& -120 & -120& 120 & 120& 120 & 120\\
\hline

\end{tabular}
\label{tabf54}
\end{table}



\subsection{Generalized flag manifolds of the classical Lie group $SO(2n)$}\label{sec2}

In this section we will study the almost complex geometry of several flag manifolds of the classical Lie group $SO(2n)$. We call this manifolds of flag manifolds of $D_{l}$ type. We will compute the Chern numbers of the invariant almost complex structures and use these numbers in order to distinguish them. In some cases, these numbers provide a complete classification of such structures, up to conjugation and equivalence.   

\subsubsection{The flag manifold $\mathbb{F}_{D}(3;1,2)=SO(6)/U(1)\times U(2)$}
Consider the $10$-dimensional flag manifold $\mathbb{F}_{D}(3;1,2)=SO(6)/U(1)\times U(2)$ with  Euler characteristic $\chi(\mathbb{F}_{D}(3;1,2))=12$.

We will use a set of $T-$roots given by Definition \ref{t-root} and the Proposition \ref{142} in order to describe the irreducible isotropic summands of $\mathbb{F}_{D}(3;1,2)$. 

The Cartan sub-algebra of $\got{h}$ of $\got{so}(6)$ is given by
\begin{equation}\label{cartan-so6}
\got{h}=\left(\begin{array}{cl}
\Lambda&\\
&-\Lambda
\end{array}\right),
\ \ \mbox{ where } \ \Lambda=\{\diag(x,y,z):x,y,z\in \mathbb{C}\},
\end{equation}

and complementary positive roots are: 
$$\Pi^{+}=\{x+y,x+z,y+z,x-y,x-z\}.$$

Consider the sub-algebra $\got{t}$ of $\got{h}$ given by
$$
\got{t}=\{\diag(d_{1},-d_{1},d_{2},-d_{2},d_{2},-d_{2})\}.
$$

The restriction of the complementary roots to the sub-algebra $\got{t}$ is the set of $T$-roots:
$$
R_{T}=\{\pm(d_{1}+d_{2}),\pm 2d_{2},\pm(d_{1}-d_{2})\}.
$$
Therefore,
$$
\begin{array}{ccl}
\got{m}&=&\got{m}_{\pm(d_{1}+d_{2})}\oplus \got{m}_{\pm 2d_{2}}\oplus \got{m}_{\pm(d_{1}-d_{2})}\\
&=&\got{m}_{1}\oplus\got{m}_{2}\oplus\got{m}_{3},
\end{array}
$$
where, $\dim_{\mathbb{C}}(\got{m}_{1})=\dim_{\mathbb{C}}(\got{m}_{3})=2$ and $\dim_{\mathbb{C}}(\got{m}_{2})=1$.

By proposition \ref{numest} $\mathbb{F}_{D}(3;1,2)$ admits $4$ invariant almost complex structures, up to conjugation:
$$
\begin{array}{ccccccc}
J_{1}=(+,+,+)&&&&&&J_{3}=(+,+,-)\\
J_{2}=(-,+,+)&&&&&&J_{4}=(+,-,+).
\end{array}
$$
Using the same techniques as in the previous section one can check that $J_{1}$, $J_{3}$ and $J_{4}$ are integrable.

The cohomology ring of $\mathbb{F}_{D}(3;1,2)$ is given by
$$
H^{*}(\mathbb{F}_{D}(3;1,2),\mathbb{R})=\dfrac{\mathbb{R}[x,y,z]}{\langle s_{1},s_{2},e_{3}\rangle},
$$
where $s_{k}=x^{2k}+y^{2k}+z^{2k}$ and $e_{3}=xyz$,  with $x,y,z$ in degree $2$ and $k=1,2$.

The class $3(y^{2}z^{3}-yz^{4})$ is the top Chern class in $H^{10}(M,\mathbb{R})$. 

The Chen numbers of the almost complex manifolds $(\mathbb{F}_{D}(3;1,2),J_{i})$, $i=1,2,3,4$ are in the table \ref{tabso1}.

\begin{table}[h] 

\centering

\caption{Chern numbers of $SO(6)/U(1)\times U(2)$} 

\begin{tabular}{|c|c|c|c|c|} 

\hline 

&$(+,+,+)$ & $(-,+,+)$ & $(+,+,-)$ & $(+,-,+)$\\ 

\hline
\hline

$c_{5}$& 12 & 12 & 12 & -12\\
\hline
$c_{1}^{5}$& 4500 & -20 & 4860 & -4500 \\
\hline
$c_{1}^{3}c_{2}$ & 2148 & -4 & 2268 & -2148\\ 
\hline
$c_{1}^{2}c_{3}$ & 612 & 20 & 612 & -612\\
\hline
$c_{1}c_{4}$ & 108 & 12 & 108 & -108\\
\hline
$c_{1}c_{2}^{2}$ & 1028 & -4 & 1068 & -1028\\
\hline 
$c_{2}c_{3}$ & 292 & 4 & 292 & -292\\
\hline
\end{tabular}
\label{tabso1}
\end{table}

\begin{proposition}
The generalized flag manifold $SO(6)/U(1)\times U(2)$ admits at least $3$ invariant almost complex structures up to equivalence and conjugation.
\end{proposition}

By Hirzebruch-Riemann-Roch Theorem we have 
$$
\sum_{q=0}^{5}(-1)^{q}h^{0,q}=(1/1440)(-c_{1}c_{4}+c_{1}^{2}c_{3}+3c_{1}c_{2}^{2}-c_{1}^{3}c_{2}).
$$

Since the cohomology ring of $\mathbb{F}_{D}(3;1,2)$ is of type $(p,p)$, $\sum_{q=0}^{5}(-1)^{q}h^{0,q}=1$, it holds the following relation between the Chern numbers:
$$
1440=-c_{1}c_{4}+c_{1}^{2}c_{3}+3c_{1}c_{2}^{2}-c_{1}^{3}c_{2}.
$$

%
%

\subsubsection{The flag manifold $\mathbb{F}_{D}(4;1,3)=SO(8)/U(1)\times U(3)$}

Consider the $18$-dimensional flag manifold $\mathbb{F}_{D}(4;1,3)=SO(8)/U(1)\times U(3)$ whose Euler characteristic is $\chi(\mathbb{F}_{D}(4;1,3))=32$.


The Cartan sub-algebra $\got{h}$ of $\got{so}(8)$ is
$$
\got{h}=\left(\begin{array}{cl}
\Lambda&\\
&-\Lambda
\end{array}\right),
\ \ where \ \Lambda=\{\diag(x,y,z,w):x,y,z,w\in \mathbb{C}\}.
$$

The complementary roots are: 
$$
\Pi^{+}=\{x_{1}+x_{2},x_{1}+x_{3},x_{1}+x_{4},x_{2}+x_{3},x_{2}+x_{4},
x_{3}+x_{4},x_{1}-x_{2},x_{1}-x_{3},x_{1}-x_{4}\}.
$$

Consider the sub-algebra $\got{t}$ of $\got{h}$ given by
$$
\got{t}=\{\diag(d_{1},-d_{1},d_{2},-d_{2},d_{2},-d_{2},d_{2},-d_{2})\}.
$$

The restriction of the complementary roots to the sub-algebra $\got{t}$ is the set of $T$-roots:
$$
R_{T}=\{\pm(d_{1}+d_{2}),\pm 2d_{2},\pm(d_{1}-d_{2})\}.
$$
Therefore,
$$
\begin{array}{ccl}
\got{m}&=&\got{m}_{\pm(d_{1}+d_{2})}\oplus \got{m}_{\pm 2d_{2}}\oplus \got{m}_{\pm(d_{1}-d_{2})}\\
&=&\got{m}_{1}\oplus\got{m}_{2}\oplus\got{m}_{3}.
\end{array}
$$

By proposition \ref{142},
$$
\begin{array}{l}
\got{m}_{1}=\got{g}_{x_{1}+x_{2}}\oplus \got{g}_{-(x_{1}+x_{2})}\oplus \got{g}_{x_{1}+x_{3}}\oplus \got{g}_{-(x_{1}+x_{3})}\oplus \got{g}_{x_{1}+x_{4}}\oplus \got{g}_{-(x_{1}+x_{4})},\\
\got{m}_{2}=\got{g}_{x_{2}+x_{3}}\oplus \got{g}_{-(x_{2}+x_{3})}\oplus \got{g}_{x_{2}+x_{4}}\oplus \got{g}_{-(x_{2}+x_{4})}\oplus \got{g}_{x_{3}+x_{4}}\oplus \got{g}_{-(x_{3}+x_{4})},\\
\got{m}_{3}=\got{g}_{x_{1}-x_{2}}\oplus \got{g}_{-(x_{1}-x_{2})}\oplus \got{g}_{x_{1}-x_{3}}\oplus \got{g}_{-(x_{1}-x_{3})}\oplus \got{g}_{x_{1}-x_{4}}\oplus \got{g}_{-(x_{1}-x_{4})},
\end{array}
$$

and $\dim_{\mathbb{C}}(\got{m}_{1})=\dim_{\mathbb{C}}(\got{m}_{2})=\dim_{\mathbb{C}}(\got{m}_{3})=3$.

According to proposition \ref{numest} $\mathbb{F}_{D}(4;1,3)$ admits $4$ invariant almost complex structures up to conjugation, namely,
$$
\begin{array}{ccccccc}
J_{1}=(+,+,+)&&&&&&J_{3}=(+,+,-)\\
J_{2}=(-,+,+)&&&&&&J_{4}=(+,-,+).
\end{array}
$$

The cohomology ring of  $\mathbb{F}_{D}(4;1,3)$ is
$$
H^{*}(\mathbb{F}_{D}(4;1,3),\mathbb{R})=\dfrac{\mathbb{R}[x,y,z,w]}{\langle s_{1},s_{2},s_{3},e_{4}\rangle},
$$
where $s_{k}=x^{2k}+y^{2k}+z^{2k}+w^{2k}$ and $e_{4}=xyzw$,  with $x,y,z,w$ in degree $2$ and $k=1,2,3$.

The Chern numbers of the almost complex structures $(\mathbb{F}_{D}(4;1,3),J_{i})$ are in the table \ref{tabso21}.

\begin{table}[h] 

\centering

\caption{Chern numbers of $SO(8)/U(1)\times U(3)$} 

\begin{tabular}{|c|c|c|c|c|} 

\hline 

&$(+,+,+)$ & $(-,+,+)$ & $(+,+,-)$ & $(+,-,+)$\\ 

\hline
\hline
$c_{9}$ & 32 & 32 & 32 & 32\\
\hline
$c_{1}c_{8}$ & -96 & 96 & 96 & 0\\
\hline
$c_{2}c_{7}$ & -786 & 786 & 786 & -6\\
\hline
$c_{1}^{2}c_{7}$ & -1632 & 1632 & 1632 & 0\\
\hline
$c_{3}c_{6}$ & -2958 & 2958 & 2958 & 6\\
\hline
$c_{1}c_{2}c_{6}$ & -9792 & 9792 & 9792 & 0\\
\hline
$c_{1}^{3}c_{6}$ & -20352& 20352 & 20352 & 0\\
\hline
$c_{4}c_{5}$ & -5592 & 5592 & 5592 & -24\\
\hline
$c_{1}c_{3}c_{5}$ & -26976 & 26976 & 26976 & 0\\
\hline
$c_{2}^{2}c_{5}$ & -43128 & 43128 & 43128 & -24\\
\hline
$c_{1}^{2}c_{2}c_{5}$ & -89856 & 89856 & 89856 & 0\\
\hline
$c_{1}^{4}c_{2}c_{3}$ & -2974464 & 2974464 & 2974464 & 0\\
\hline
$c_{1}^{4}c_{5}$ & -187392 & 187392 & 187392 & 0\\
\hline
$c_{1}c_{4}^{2}$ & -37440 & 37440 & 37440 & 0\\
\hline
$c_{2}c_{3}c_{4}$ & -87138 & 87138 & 87138 & -6\\
\hline
$c_{1}^{2}c_{3}c_{4}$ & -181728 & 181728 & 181728 & 0\\
\hline
$c_{1}c_{2}^{2}c_{4}$ & -291168 & 291168 & 291168 & 0\\
\hline
$c_{1}^{3}c_{2}c_{4}$ & -608256 & 608256 & 608256 & 0\\
\hline
$c_{1}^{5}c_{4}$ & -1271808 & 1271808 & 1271808 & 0\\
\hline
$c_{3}^{3}$ & -126800 & 126800 & 126800 & 48\\
\hline
$c_{1}c_{2}c_{3}^{2}$& -423936 & 423936 & 423936 & 0 \\
\hline
$c_{1}^{3}c_{3}^{2}$ & -885888 & 885888 & 885888 & 0\\ 
\hline
$c_{2}^{3}c_{3}$ & -679698 & 679698 & 679698 & -6\\
\hline
$c_{1}^{2}c_{2}^{2}c_{3}$ & -1421280 & 1421280 & 1421280 & 0\\
\hline
$c_{1}^{6}c_{3}$ & -6230016 & 6230016 & 6230016 & 0\\
\hline 
$c_{1}c_{2}^{4}$ & -2280960 & 2280960 & 2280960 & 0\\
\hline
$c_{1}^{3}c_{2}^{3}$ & -4776192 & 4776192 & 4776192 & 0\\
\hline
$c_{1}^{5}c_{2}^{2}$ & -10008576 & 10008576 & 10008576 & 0\\
\hline
$c_{1}^{7}c_{2}$ & -20987904 & 20987904 & 20987904 & 0\\
\hline
$c_{1}^{9}$ & -44040192 & 44040192 & 44040192 & 0\\
\hline
\end{tabular}
\label{tabso21}
\end{table}

\begin{proposition}
The generalized flag manifold $SO(8)/U(1)\times U(3)$ admit at least $2$ invariant almost complex structures up to conjugation and equivalence. 
\end{proposition}

Applying Hirzebruch-Riemann-Roch Theorem we have
$$
\begin{array}{ccl}
\sum_{q=0}^{9}(-1)^{q}h^{0,q}&=&(1/7257600)(-3c_{1}^{7}c_{2}+21c_{1}^{5}c_{2}^{2}-42c_{1}^{3}c_{2}^{3}+26c_{1}^{3}c_{2}c_{4}+3c_{1}^{6}c_{3}\\
&-&13c_{1}^{2}c_{3}c_{4}-3c_{1}^{5}c_{4}+21c_{1}c_{2}^{4}-34c_{1}c_{2}^{2}c_{4}+5c_{1}c_{4}^{2}+
3c_{1}^{4}c_{5}\\
&-&29c_{1}^{4}c_{2}c_{3}+50c_{1}^{2}c_{2}^{2}c_{3}+8c_{1}^{3}c_{3}^{2}-8c_{1}c_{2}c_{3}^{2}-
16c_{1}^{2}c_{2}c_{5}+3c_{1}c_{3}c_{5}\\
&-&3c_{1}^{3}c_{6}+13c_{1}c_{2}c_{6}+3c_{1}^{2}c_{7}-3c_{1}c_{8}).
\end{array}
$$

Since the cohomology of $(\mathbb{F}_{D}(4;1,3),J)$ (with $J$ integrable) is of type $(p,p)$, $\sum_{q=0}^{9}(-1)^{q}h^{0,q}=1$ it holds the following relation between the Chern numbers:

$$
\begin{array}{ccl}
7257600&=&-3c_{1}^{7}c_{2}+21c_{1}^{5}c_{2}^{2}-42c_{1}^{3}c_{2}^{3}+26c_{1}^{3}c_{2}c_{4}+
3c_{1}^{6}c_{3}-13c_{1}^{2}c_{3}c_{4}-3c_{1}^{5}c_{4}\\&+&
21c_{1}c_{2}^{4}-34c_{1}c_{2}^{2}c_{4}+5c_{1}c_{4}^{2}+
3c_{1}^{4}c_{5}-29c_{1}^{4}c_{2}c_{3}+50c_{1}^{2}c_{2}^{2}c_{3}
+8c_{1}^{3}c_{3}^{2}\\&-&8c_{1}c_{2}c_{3}^{2}-
16c_{1}^{2}c_{2}c_{5}+3c_{1}c_{3}c_{5}-3c_{1}^{3}c_{6}+13c_{1}c_{2}c_{6}
+3c_{1}^{2}c_{7}-
3c_{1}c_{8}.
\end{array}
$$

\subsubsection{The flag manifold $\mathbb{F}_{D}(4;4)=SO(8)/U(4)$}

Consider the 12-dimensional generalized flag manifold $\mathbb{F}_{D}(4;4)=SO(8)/U(4)$ with Euler characteristic 8.

According \cite{livro}, a Cartan sub-algebra $\got{h}$ of $\got{so}(8)$ is given by 
$$
\got{h}=\left(\begin{array}{cc}
\Lambda&\\
&-\Lambda
\end{array}\right), \ \ \mbox{where} \ \ \Lambda=\{\diag(x,y,z,w):x,y,z,w\in \mathbb{C}\}.
$$
The set of positive roots of $\mathbb{F}_{D}(4;4)$ is
$$
\Pi^{+}=\{x+y,x+z,x+w,y+z,y+w,z+w\}.
$$

Consider the following sub-algebra $\got{t}$ of $\got{h}$ given by 
$$
\got{t}=\diag\{d_{1},d_{1},d_{1},d_{1},-d_{1},-d_{1},-d_{1},-d_{1}\}.
$$
The restriction of the roots $\Pi$ to $\got{t}$ is the set
$$
R_{T}=\{\pm 2d_{1}\}
$$
and consequently the isotropy representation of $\mathbb{F}_{D}(4;4)$ decompose in just one component
$$
\got{m}=\got{m}_{\pm (2d_{1})},
$$
that is, $\mathbb{F}_{D}(4;4)$ is isotropically irreducible. Actually, $\mathbb{F}_{D}(4;4)$ is a Hermitian symmetric space.

According the proposition \ref{numest}, $\mathbb{F}_{D}(4;4)$ admits just one invariant almost complex structure, up to conjugation.

Let us compute the Chern number of $\mathbb{F}_{D}(4;4)=SO(8)/U(4)$ with respect to this invariant complex structure. Consider the cohomology ring 
$$
H^{*}(\mathbb{F}_{D}(4;4),\mathbb{R})=\dfrac{\mathbb{R}[x,y,z,w]}{\langle s_{1},s_{2},s_{3},e_{4}\rangle},
$$
where $s_{k}=x^{2k}+y^{2k}+z^{2k}+w^{2k}$ and $e_{4}=xyzw$,  with $x,y,z,w$ in degree $2$ and $k=1,2,3$.

The Chern numbers of $(SO(8)/U(4),J)$, where $J$ is the canonical complex structure, are listed in the next table. 
$$
\begin{array}{|c|}
\hline
\mbox{Chern numbers of} \ (\mathbb{F}_{D}(4;4),J) \ \mbox{com} \ J=(+)\\
\hline
\hline
\begin{array}{lccccccccl}
c_{6}=8&&&&&&&&&c_{1}c_{5}=144 \\ \hline
c_{2}c_{4}=704&&&&&&&&&c_{1}^{2}c_{4}=1584\\ \hline
c_{3}^{2}=1152&&&&&&&&&c_{1}c_{2}c_{3}=4608\\ \hline
c_{1}^{3}c_{3}=10368&&&&&&&&&c_{2}^{3}=8192\\ \hline
c_{1}^{2}c_{2}^{2}=18432&&&&&&&&&c_{1}^{4}c_{2}=41472\\ \hline
c_{1}^{6}=93312&&&&&&&&\\
\end{array}\\
\hline
\end{array}
$$

According to the Hirzebruch-Riemann-Roch Theorem the Chern numbers must satisfy the relation:
$$
\begin{array}{ccl}
\sum_{q=1}^{5}(-1)^{5}h^{0,q}&=&(1/60480)(2c_{6}-2c_{1}c_{5}-9c_{2}c_{4}-5c_{1}^{2}c_{4}-c_{3}^{2}+11c_{1}c_{2}c_{3}+5c_{1}^{3}c_{3}\\
&+&10c_{2}^{3}+11c_{1}^{2}c_{2}^{2}-12c_{1}^{4}c_{2}+2c_{1}^{6}).
\end{array}
$$
Therefore,
$$
\begin{array}{ccl}
60480&=&2c_{6}-2c_{1}c_{5}-9c_{2}c_{4}-5c_{1}^{2}c_{4}-c_{3}^{2}+11c_{1}c_{2}c_{3}+5c_{1}^{3}c_{3}\\
&+&10c_{2}^{3}+11c_{1}^{2}c_{2}^{2}-12c_{1}^{4}c_{2}+2c_{1}^{6}.
\end{array}
$$

\subsection{Generalized flag manifolds of classical Lie groups $SO(2n+1)$ and $Sp(n)$} \label{sec3}
In this section we compute the Chern numbers of some generalized flag manifolds of the classical Lie groups $SO(2n+1)$ and $Sp(n)$. 

\subsubsection{The flag manifold $\mathbb{F}_{B}(2;1,1)=SO(5)/T$}

Consider the full flag manifold $\mathbb{F}_{B}(2;1,1)=SO(5)/T$ with real dimension $8$ and Euler characteristic $\chi(\mathbb{F}_{B}(2;1,1))=\vert \mathcal{W}_{SO(5)}\vert= 2^{2}.2=8$, where $\vert \mathcal{W}_{SO(5)}\vert$ is the order of the Weyl group of $SO(5)$.

Let $\got{h}=\left(\begin{array}{ccc}
0&&\\
&\Lambda&\\
&&-\Lambda
\end{array}\right)$, with $\Lambda=\left(\begin{array}{cc}
x&\\
&y
\end{array}\right)$, be a Cartan sub-algebra of $\got{so}(5)$.

Consider the following maps 
$$
\lambda_{1}:\Lambda=\diag\{x,y\}\longrightarrow x \ \ and \ \ 
\lambda_{2}:\Lambda=\diag\{x,y\}\longrightarrow y.
$$

The roots of $\mathbb{F}_{B}(2;1,1)$ are described as follow: $\alpha_{1}=\lambda_{1}$, $\alpha_{2}=\lambda_{2}$, $\alpha_{1}-\alpha_{2}=\lambda_{1}-\lambda_{2}$ and $\alpha_{1}+\alpha_{2}=\lambda_{1}+\lambda_{2}$. Therefore the isotropy representation of $\mathbb{F}_{B}(2;1,1)$ admits $4$ isotropy summands:
$$
\begin{array}{ccl}
\got{m}&=&\got{m}_{1}\oplus \got{m}_{2}\oplus \got{m}_{3}\oplus \got{m}_{4}\\
&=&\got{g}_{\pm\alpha_{1}}\oplus \got{g}_{\pm\alpha_{2}}\oplus \got{g}_{\pm(\alpha_{1}-\alpha_{2})}\oplus \got{g}_{\pm(\alpha_{1}+\alpha_{2})}. 
\end{array}
$$

By proposition \ref{numest}, $\mathbb{F}_{B}(2;1,1)$ has $2^3=8$ invariant almost complex structures up conjugation. The canonical complex structure $J$ is represented by $J=(+,+,+,+)$.


The Chern classes of $\mathbb{F}_{B}(2;1,1)$ are given by
$$
\begin{array}{lccccccccl}
c_{1}(\mathbb{F}_{B}(2;1,1))=3x+y&&&&&&&&&
c_{2}(\mathbb{F}_{B}(2;1,1))=3xy-4y^{2}\\
c_{3}(\mathbb{F}_{B}(2;1,1))=-2(xy^{2}+2y^{3})&&&&&&&&&
c_{4}(\mathbb{F}_{B}(2;1,1))=-2xy^{3}.
\end{array}
$$

The Chern numbers of $\mathbb{F}_{B}(2;1,1)$ are given in the next table:
$$
\begin{array}{|c|}
\hline
\mbox{Chern numbers of } \ (\mathbb{F}_{B}(2;1,1),J)\\
\hline
\hline
\begin{array}{lccccccccl}
c_{1}^{4}=-384&&&&&&&&&c_{2}^2=-96\\ \hline
c_{1}^{2}c_{2}=-192&&&&&&&&&c_{4}=-8\\ \hline
c_{1}c_{3}=-56.&&&&&&&&&\\ \hline
\end{array}\\
\hline
\end{array}
$$

The relations between the Chern numbers are given by the Hirzebruch-Riemann-Roch Theorem
$$
\begin{array}{ccl}
720&=&-c_{4}+c_{1}c_{3}+3c_{2}^{2}+4c_{1}^{2}c_{2}-c_{1}^{4}\\
&=&8-56+3(-96)+4(-192)+384\\
&=&-1112+392\\
&=&720.
\end{array}
$$

\subsubsection{The flag manifold $\mathbb{F}_{B}(3;3)=SO(7)/U(3)$}

Let $\mathbb{F}_{B}(3;3)=SO(7)/U(3)$ be a 12-dimensional flag manifold with Euler characteristic $\chi(\mathbb{F}_{B}(3;3))=2^{3}3!/3!=2^{3}=8$.

Let $\got{h}=\left(\begin{array}{ccc}
0&&\\
&\Lambda&\\
&&-\Lambda
\end{array}\right)$, with $\Lambda=\left(\begin{array}{ccc}
x&&\\
&y&\\
&&z
\end{array}\right)$, be a Cartan sub-algebra of $\got{so}(7)$.

Consider the linear functionals given by 
$$
\begin{array}{l}
\lambda_{1}:\Lambda=\diag\{x,y,z\}\longrightarrow x,\\ 
\lambda_{2}:\Lambda=\diag\{x,y,z\}\longrightarrow y,\\
\lambda_{3}:\Lambda=\diag\{x,y,z\}\longrightarrow z.
\end{array}
$$
The positive complementary roots of  $\mathbb{F}_{B}(3;3)$ are the following: $\alpha_{1}=\lambda_{1}$, $\alpha_{2}=\lambda_{2}$, $\alpha_{3}=\lambda_{3}$, $\alpha_{4}=\lambda_{1}+\lambda_{2}$, $\alpha_{5}=\lambda_{1}+\lambda_{3}$ and $\alpha_{6}=\lambda_{2}+\lambda_{3}$. 

Let $\got{t}$ be the sub-algebra defined by
$
\got{t}=\left(\begin{array}{ccc}
0&&\\
&\Lambda&\\
&&-\Lambda
\end{array}\right)$, with $\Lambda=\left(\begin{array}{ccc}
d_{1}&&\\
&d_{1}&\\
&&d_{1}
\end{array}\right).
$
The set of T-roots is given by, 
$$
R_{T}=\{\pm d_{1},\pm 2d_{1}\}.
$$ 
By proposition  \ref{142} we have
$$
\begin{array}{ccl}
\got{m}&=&\got{m}_{1}\oplus \got{m}_{2}\\
&=&\got{m}_{\pm 2d_{1}}\oplus \got{m}_{\pm d_{1}}.
\end{array}
$$

By proposition \ref{numest}, $\mathbb{F}_{B}(3;3)$ admits 2 invariant almost complex structures, up to conjugation.

The cohomology ring of $\mathbb{F}_{B}(3;3)$ is given by
$$
H^{*}(\mathbb{F}_{B}(3;3),\mathbb{R})=\dfrac{R[x,y,z]}{\langle s_{1},s_{2},s_{3}\rangle},
$$
where $s_{k}=x^{2k}+y^{2k}+z^{2k}$,  with $x,y,z$ in degree $2$ and $k=1,2,3$.

The Chern numbers of $(\mathbb{F}_{B}(3;3),J)$, where $J$ is the canonical almost complex structure are given by:
$$
\begin{array}{|c|}
\hline
\mbox{Chern numbers of} \ (\mathbb{F}_{B}(3;3),J)\\
\hline
\hline
\begin{array}{cccccccccc}
c_{1}c_{5}=144&&&&&&&&&c_{1}^2c_{2}^{2}=18432\\
c_{1}^{2}c_{4}=1584&&&&&&&&&c_{1}c_{2}c_{3}=4608\\
c_{1}^{3}c_{3}=10368&&&&&&&&&c_{2}^{3}=8192\\
c_{1}^{4}c_{2}=41472&&&&&&&&&c_{2}c_{4}=704\\
c_{1}c_{2}^{2}=6144&&&&&&&&&c_{3}^{2}=1152\\
c_{6}=8&&&&&&&&&\\
\end{array}\\
\hline
\end{array}.
$$

The relations between the Chern numbers are given by the Hirzebruch-Riemann-Roch Theorem 

$$
\begin{array}{ccl}
60480&=&2c_{6}-2c_{1}c_{5}-9c_{2}c_{4}-5c_{1}^{2}c_{4}-c_{3}^{2}+11c_{1}c_{2}c_{3}+5c_{1}^{3}c_{3}\\&+&10c_{2}^{3}+11c_{1}^{2}c_{2}^{2}-12c_{1}^{4}c_{2}+2c_{1}^{6}\\
\end{array}
$$

\subsubsection{The flag manifold  $\mathbb{F}_{C}(2;1,1)=Sp(2)/T$}

Consider the 8-dimensional flag manifold  $\mathbb{F}_{C}(2;1,1)=Sp(2)/T$ with Euler characteristic $\chi(\mathbb{F}_{C}(2;1,1))=\vert \mathcal{W}_{SO(5)}\vert= 2^{2}.2=8$.


Let  $\got{h}=\left(\begin{array}{cc}
\Lambda&\\
&-\Lambda
\end{array}\right)$, with $\Lambda=\left(\begin{array}{cc}
x&\\
&y
\end{array}\right)$, be a Cartan sub-algebra of $\got{sp}(2)$.

Consider the linear functionals 
$$
\lambda_{1}:\Lambda=\diag\{x,y\}\longrightarrow x \ \ and \ \ 
\lambda_{2}:\Lambda=\diag\{x,y\}\longrightarrow y.
$$
The roots of  $\mathbb{F}_{C}(2;1,1)$ are given by: $2\alpha_{1}=2\lambda_{1}$, $2\alpha_{2}=2\lambda_{2}$, $\alpha_{1}-\alpha_{2}=\lambda_{1}-\lambda_{2}$ and $\alpha_{1}+\alpha_{2}=\lambda_{1}+\lambda_{2}$. Therefore the isotropy representation of $\mathbb{F}_{C}(2;1,1)$ admits $4$ isotropy summands:
$$
\begin{array}{ccl}
\got{m}&=&\got{m}_{1}\oplus \got{m}_{2}\oplus \got{m}_{3}\oplus \got{m}_{4}\\
&=&\got{g}_{\pm 2\alpha_{1}}\oplus \got{g}_{\pm 2\alpha_{2}}\oplus \got{g}_{\pm(\alpha_{1}-\alpha_{2})}\oplus \got{g}_{\pm(\alpha_{1}+\alpha_{2})}. 
\end{array}
$$

By proposition \ref{numest} $\mathbb{F}_{C}(2;1,1)$ admits $2^{3}=8$ invariant almost complex structures, up to conjugation. Let $J$ be the canonical invariant complex structure, that is, $J=(+,+,+,+)$.

Consider the cohomology ring of $\mathbb{F}_{C}(2;1,1)$ 
$$
H^{*}(\mathbb{F}_{C}(2;1,1),\mathbb{R})=\dfrac{R[x,y]}{\langle s_{1},s_{2}\rangle},
$$
where $s_{k}=x^{2k}+y^{2k}$,  with $x,y$ in degree $2$ and $k=1,2$.

The Chern classes of  $\mathbb{F}_{C}(2;1,1)$ are given by
$$
\begin{array}{lccccccccl}
c_{1}(M)=4x+2y&&&&&&&&&
c_{2}(M)=8xy-6y^{2}\\
c_{3}(M)=-4(xy^{2}+3y^{3})&&&&&&&&&
c_{4}(M)=-8xy^{3}.
\end{array}
$$


The Chern numbers of  $\mathbb{F}_{C}(2;1,1)$ are listed in the following table:
$$
\begin{array}{|c|}
\hline
\mbox{Chern numbers of } \ (\mathbb{F}_{C}(2;1,1),J)\\
\hline
\hline
\begin{array}{cccccccccc}
c_{1}^{4}=-384&&&&&&&&&c_{2}^2=-96\\
c_{1}^{2}c_{2}=-192&&&&&&&&&c_{4}=-8\\
c_{1}c_{3}=-56.&&&&&&&&&\\
\end{array}\\
\hline
\end{array}.
$$
\begin{remark}
Note that Chern numbers of $SO(5)/T$ and $Sp(2)/T$ are the same, but the Chern classes of these two manifolds are different. 
\end{remark}

\subsubsection{The flag manifold $\mathbb{F}_{C}(3;1,1,1)=Sp(3)/T$}

Consider the 18-dimensional full flag manifold $\mathbb{F}_{C}(3;1,1,1)=Sp(3)/T$ with Euler characteristic $\chi(\mathbb{F}_{C}(3;1,1,1))=2^{3}3!=48$. 

Let $\got{h}=\left(\begin{array}{cc}
\Lambda&\\
&-\Lambda
\end{array}\right)$, with $\Lambda=\left(\begin{array}{ccc}
x&&\\
&y&\\
&&z
\end{array}\right)$, be a Cartan sub-algebra of $\got{sp}(3)$.

Consider the linear functionals given by
$$
\begin{array}{l}
\lambda_{1}:\Lambda=\diag\{x,y,z\}\longrightarrow x,\\
\lambda_{2}:\Lambda=\diag\{x,y,z\}\longrightarrow y,\\
\lambda_{3}:\Lambda=\diag\{x,y,z\}\longrightarrow z.
\end{array}
$$
The roots of $\mathbb{F}_{C}(3;1,1,1)$ are given by: $2\alpha_{1}=2\lambda_{1}$, $2\alpha_{2}=2\lambda_{2}$, $2\alpha_{3}=2\lambda_{3}$, $\alpha_{1}-\alpha_{2}=\lambda_{1}-\lambda_{2}$, $\alpha_{1}-\alpha_{3}=\lambda_{1}-\lambda_{3}$, $\alpha_{2}-\alpha_{3}=\lambda_{2}-\lambda_{3}$, $\alpha_{1}+\alpha_{2}=\lambda_{1}+\lambda_{2}$, $\alpha_{1}+\alpha_{3}=\lambda_{1}+\lambda_{3}$ and $\alpha_{2}+\alpha_{3}=\lambda_{2}+\lambda_{3}$. The isotropy representation of $\mathbb{F}_{C}(3;1,1,1)$ admits 9 isotropy summands:
$$
\begin{array}{ccl}
\got{m}&=&\got{m}_{1}\oplus \got{m}_{2}\oplus \got{m}_{3}\oplus \got{m}_{4}\oplus\got{m}_{5}\oplus \got{m}_{6}\oplus \got{m}_{7}\oplus \got{m}_{8}\oplus \got{m}_{9}\\
&=&\got{g}_{\pm 2\alpha_{1}}\oplus \got{g}_{\pm 2\alpha_{2}}\oplus \got{g}_{\pm(2\alpha_{3})}\oplus \got{g}_{\pm(\alpha_{1}-\alpha_{2})}\oplus \got{g}_{\pm(\alpha_{1}-\alpha_{3})}\oplus \got{g}_{\pm(\alpha_{2}-\alpha_{3})}\\
&\oplus & \got{g}_{\pm(\alpha_{1}+\alpha_{2})}\oplus \got{g}_{\pm(\alpha_{1}+\alpha_{3})}\oplus \got{g}_{\pm(\alpha_{2}+\alpha_{3})}. 
\end{array}
$$
Therefore $\mathbb{F}_{C}(3;1,1,1)$ admits $2^{8}$ invariant almost complex structures up to conjugation. Let us consider the canonical invariant almost complex structure $J=(+,+,+,+,+,+,+,+,+)$ and compute their Chern numbers.

The cohomology ring of $\mathbb{F}_{C}(3;1,1,1)$ is given by
$$
H^{*}(\mathbb{F}_{C}(3;1,1,1),\mathbb{R})=\dfrac{R[x,y,z]}{\langle s_{1},s_{2},s_{3}\rangle},
$$
where $s_{k}=x^{2k}+y^{2k}+z^{2k}$,  with $x,y,z$ in degree $2$ and $k=1,2,3$.

The Chern numbers of $\mathbb{F}_{C}(3;1,1,1)$ are listed in the table \ref{tabsp31}.

\begin{table}[h] 

\centering

\caption{Chern numbers of $(Sp(3)/T,\, J)$ $J$ invariant complex structure} 

\begin{tabular}{|c|c|c|} 
\hline 

$(+,+,+)$& & $(+,+,+)$\\ 

\hline
\hline
$c_{9}=-48$ &  & $c_{1}c_{4}^{2}=-270208$ \\
\hline
$c_{1}c_{8}=-1056$ & & $c_{2}c_{3}c_{4}=-578480$\\
\hline
$c_{2}c_{7}=-7696$ &  & $c_{1}^{2}c_{3}c_{4}=-1156960$\\
\hline
$c_{1}^{2}c_{7}=-15392$ &  & $c_{1}c_{2}^{2}c_{4}=-1773504$\\
\hline
$c_{3}c_{6}=-26096$ &  & $c_{1}^{3}c_{2}c_{4}=-3547008$\\
\hline
$c_{1}c_{2}c_{6}=-80272$ &  & $c_{1}^{5}c_{4}=-7094016$\\
\hline
$c_{1}^{3}c_{6}=-160544$ & & $c_{3}^{3}=-807072$\\
\hline
$c_{4}c_{5}=-46768$ &  & $c_{1}c_{2}c_{3}^{2}=-2473376$\\
\hline
$c_{1}c_{3}c_{5}=-201040$ &  & $c_{1}^{3}c_{3}^{2}=-4946752$\\
\hline
$c_{2}^{2}c_{5}=-308544$ &  & $c_{2}^{3}c_{3}=-3789792$\\
\hline
$c_{1}^{2}c_{2}c_{5}=-617088$ &  & $c_{1}^{2}c_{2}^{2}c_{3}=-7579584$ \\
\hline
$c_{1}^{4}c_{2}c_{3}=-15159168$ &  & $c_{1}^{6}c_{3}=-30318336$\\
\hline
$c_{1}^{4}c_{5}=-1234176$ &  & $c_{1}c_{2}^{4}=-11612160$\\
\hline
$c_{1}^{3}c_{2}^{3}=-23224320$ &  & $c_{1}^{7}c_{2}=-92897280$\\
\hline
$c_{1}^{5}c_{2}^{2}=-46448640$ &  & $c_{1}^{9}=-185794560$.\\
\hline
\end{tabular}
\label{tabsp31}
\end{table}

The relations between the Chern numbers are given by the Hirzebruch-Riemann-Roch Theorem: 
$$
\begin{array}{ccl}
7257600&=&-3c_{1}^{7}c_{2}+21c_{1}^{5}c_{2}^{2}-42c_{1}^{3}c_{2}^{3}+26c_{1}^{3}c_{2}c_{4}+
3c_{1}^{6}c_{3}-13c_{1}^{2}c_{3}c_{4}-3c_{1}^{5}c_{4}\\&+&
21c_{1}c_{2}^{4}-34c_{1}c_{2}^{2}c_{4}+5c_{1}c_{4}^{2}+
3c_{1}^{4}c_{5}-29c_{1}^{4}c_{2}c_{3}+50c_{1}^{2}c_{2}^{2}c_{3}
+8c_{1}^{3}c_{3}^{2}\\&-&8c_{1}c_{2}c_{3}^{2}-
16c_{1}^{2}c_{2}c_{5}+3c_{1}c_{3}c_{5}-3c_{1}^{3}c_{6}+13c_{1}c_{2}c_{6}
+3c_{1}^{2}c_{7}-
3c_{1}c_{8}.
\end{array}
$$

\subsection{Generalized flag manifolds of the exceptional Lie group $G_{2}$} \label{sec4}

We start with some basic facts about the Lie algebra of $G_2$.  

Consider the Cartan sub-algebra  $\got{h}$ of $\got{g}_2$ as a sub-algebra of the diagonal matrix of  $\got{sl}(3)$ and let $\lambda_{i}$ be the functional of $\got{h}$ defined by 
$$
\lambda_{i}: \diag\{a_{1},a_{2},a_{3}\}\longmapsto a_{i}.
$$

The simple roots of $G_{2}$ are $\alpha_{1}=\lambda_{1}-\lambda_{2}$ and $\alpha_{2}=\lambda_{2}$. The set of positive roots are given by  
$$
\Pi^{+}=\{\alpha_{1},\alpha_{2},\alpha_{1}+\alpha_{2},\alpha_{1}+2\alpha_{2},\alpha_{1}+3\alpha_{2},2\alpha_{1}+3\alpha_{2}\}.
$$

The maximal root of $G_{2}$ is $\mu=2\alpha_{1}+3\alpha_{2}$. For more details about the Lie algebra structure of $\mathfrak{g}_2$ see \cite{livro}. Therefore we have three flag manifolds associated to the Lie group $G_{2}$: 

\begin{enumerate}
\item $\mathbb{F}_{G}=G_2/T$, the full flag manifold. This manifold has $6$ isotropy summands.
\item $\mathbb{F}_{G_{\alpha_{1}}}=G_{2}/U(2)$ defined in terms of simple roots by $\Theta=\Sigma-\{\alpha_{1}\}$. This flag manifold has $2$ isotropy summands;

\item $\mathbb{F}_{G_{\alpha_{2}}}=G_{2}/U(2)$ defined in terms of simple roots by $\Theta=\Sigma-\{\alpha_{2}\}$.This flag manifold has $3$ isotropy summands.
\end{enumerate}

We will compute in the next sections the Chern numbers of $(\mathbb{F}_{G},J)$ (where $J$ is the canonical invariant complex structure), $\mathbb{F}_{G_{\alpha_{1}}}$ and $\mathbb{F}_{G_{\alpha_{2}}}$.

\subsubsection{The full flag manifold $\mathbb{F}_{G}=G_{2}/T$}

Let us consider the 12-dimensional full flag manifold $\mathbb{F}_{G}=G_{2}/T$. The Euler characteristic of $\mathbb{F}_{G}$ is $\chi(M)=\vert \mathcal{W}_{G_{2}}\vert=12$, where $\vert \mathcal{W}_{G_{2}}\vert$ is the order of the Weyl group of $G_{2}$.

The decomposition of the tangent space at the origin is given by
$$
\begin{array}{ccl}
\got{m}&=&\got{m_{1}}\oplus\got{m_{2}}\oplus\got{m_{3}}\oplus\got{m_{4}}\oplus\got{m_{5}}\oplus\got{m_{6}}\\
&=&\got{g}_{\alpha_{1}}\oplus\got{g}_{\alpha_{2}}\oplus\got{g}_{\alpha_{1}+\alpha_{2}}\oplus\got{g}_{\alpha_{1}+2\alpha_{2}}\oplus\got{g}_{\alpha_{1}+3\alpha_{2}}\oplus\got{g}_{2\alpha_{1}+3\alpha_{2}}.
\end{array}
$$
According to proposition \ref{numest}, $\mathbb{F}_{G}$ admits $2^{5}=32$ invariant almost complex structures up to conjugation. Furthermore, $\mathbb{F}_{G}$ has a canonical invariant complex structure compatible with the K\"ahler-Einstein metric, see \cite{grego2}.

Let us compute the Chern numbers of $(\mathbb{F}_{G},J)$, where $J$ is the canonical invariant complex structure.

The cohomology ring of $\mathbb{F}_{G}$:
$$
H^{*}(\mathbb{F}_{G},\mathbb{R})=\dfrac{\mathbb{R}[x,y,z]}{\langle s_{1},s_{2},s_{3}\rangle},
$$
where $s_{1}=x+y+z$, $s_{2}=x^{2}+y^{2}+z^{2}$ and $s_{3}=x^{6}+y^{6}+z^{6}$.

In the next table we compute the Chern numbers of $(\mathbb{F}_{G},J)$.

$$
\begin{array}{|c|}
\hline
\mbox{Chern numbers of } \ (\mathbb{F}_{G},J)\\ 
\begin{array}{lccccccl}
\hline
\hline
c_{6}=12& &&&&&& c_{1}^{6}=46080 \\
\hline
c_{1}c_{5}=192&&&&&&& c_{1}c_{2}c_{3}=3632 \\
\hline
c_{1}^{2}c_{4}=1504 &&&&&&& c_{1}^{2}c_{2}^{2}=11520 \\
\hline
c_{1}^{3}c_{3}=7264 &&&&&& & c_{1}^{4}c_{2}=23040 \\
\hline
c_{2}c_{4}=752 &&&&&&& c_{3}^{2}=1144 \\
\hline
c_{2}^{3}c_{2}=5760 &&&&&& & \\ 
\end{array}\\
\hline
\end{array}.
$$

According to the Hirzebruch-Riemann-Roch Theorem the Chern numbers must satisfy the following relations:
$$
\begin{array}{ccl}
60480&=&2c_{6}-2c_{1}c_{5}-9c_{2}c_{4}-5c_{1}^{2}c_{4}-c_{3}^{2}+11c_{1}c_{2}c_{3}+5c_{1}^{3}c_{3}\\
&+&10c_{2}^{3}+11c_{1}^{2}c_{2}^{2}-12c_{1}^{4}c_{2}+2c_{1}^{6}.
\end{array}
$$

%
%

\subsubsection{The flag manifold $\mathbb{F}_{G_{\alpha_{2}}}=G_{2}/U(2)$ with $3$ isotropic summands}

Consider the 10-dimensional flag manifold $\mathbb{F}_{G_{\alpha_{2}}}=G_{2}/U(2)$, where $\alpha_2$ is the long root of $G_2$. The Euler characteristic of this manifold is $\chi(\mathbb{F}_{G_{\alpha_{2}}})=12/2=6$.

The flag manifold $G_{\alpha_{2}}=G_{2}/U(2)$ is defined using the set of simple roots $\Theta=\Sigma-\{\alpha_{2}\}$ and has 3 isotropy summands $\got{m}_{1}$, $\got{m}_{2}$ and $\got{m}_{3}$ given by 
$$
\begin{array}{l}
\got{m}_{1}=R^{+}(\alpha_{2},1)=\{\alpha_{1}+\alpha_{2},\alpha_{2}\},\\
\got{m}_{2}=R^{+}(\alpha_{2},2)=\{\alpha_{1}+2\alpha_{2}\},\\
\got{m}_{3}=R^{+}(\alpha_{2},3)=\{\alpha_{1}+3\alpha_{2},2\alpha_{1}+3\alpha_{2}\}.
\end{array}
$$
According to proposition \ref{numest} the flag manifold $\mathbb{F}_{G_{\alpha_{2}}}$ admits $4$ invariant almost complex structures, up to conjugation:
$$
\begin{array}{ccccccc}
J_{1}=(+,+,+)&&&&&&J_{3}=(-,+,-)\\
J_{2}=(-,+,+)&&&&&&J_{4}=(+,+,-).
\end{array}
$$
By lemma \ref{order}, only $J_{1}$ is integrable.

The cohomology ring of $\mathbb{F}_{G_{\alpha_{2}}}$ is given by 
$$
H^{*}(\mathbb{F}_{G_{\alpha_{2}}},\mathbb{R})=\dfrac{\mathbb{R}[x,y,z]}{\langle s_{1},s_{2},s_{3}\rangle},
$$
where $s_{1}=x+y+z$, $s_{2}=x^{2}+y^{2}+z^{2}$ and $s_{3}=x^{6}+y^{6}+z^{6}$.

Therefore the top Chern class is given by: $c_{5}(\mathbb{F}_{G_{\alpha_{2}}})=-6z^{5}$.

The Chern numbers of the almost complex manifolds $(\mathbb{F}_{G_{\alpha_{2}}},J_{i})$, $i=1,2,3,4$, are given in the table \ref{tabg21}.

\begin{table}[h] 

\centering

\caption{Chern numbers of $G_{2}/U(2)$ with 3 isotropic summands} 

\begin{tabular}{|c|c|c|c|c|} 

\hline 

&$(+,+,+)$ & $(-,+,+)$ & $(-,+,-)$ & $(+,+,-)$\\ 

\hline
\hline

$c_{5}$& -6 & -6 & -6 & -6\\
\hline
$c_{1}^{5}$& -6250 & -486 &+486 & 2 \\
\hline
$c_{1}^{3}c_{2}$ & -2750 & -162 & +162 & -2\\ 
\hline
$c_{1}^{2}c_{3}$ & -650 & -9 & -9 & -2\\
\hline
$c_{1}c_{4}$ & -90 & -18 & -18 & 6\\
\hline
$c_{1}c_{2}^{2}$ & -1210 & -54 & +54 & 2\\
\hline 
$c_{2}c_{3}$ & -286 & -6 & -6 & 2\\
\hline
\end{tabular}
\label{tabg21}
\end{table}

According to the Hirzebruch-Riemann-Roch Theorem the Chern numbers of $(\mathbb{F}_{G_{\alpha_{2}}},J_{1})$ must satisfy the following relations: 
%
$$
1440=-c_{1}c_{4}+c_{1}^{2}c_{3}+3c_{1}c_{2}^{2}-c_{1}^{3}c_{2}.
$$

%

Analysing the table \ref{tabg21} we have the following result:

\begin{proposition}
The generalized flag manifold $\mathbb{F}_{G_{\alpha_{2}}}=G_{2}/U(2)$ admits at least $3$ invariant almost complex structures, up to conjugation and equivalence.
\end{proposition}


\subsubsection{The flag manifold $\mathbb{F}_{G_{\alpha_{1}}}=G_{2}/U(2)$ with $2$ isotropic summands}

Analogously to the previous section, consider the 10-dimensional flag manifold $\mathbb{F}_{G_{\alpha_{1}}}=G_{2}/U(2)$ with Euler characteristic $\chi(\mathbb{F}_{G_{\alpha_{1}}})=12/2=6$.

The flag manifold $G_{\alpha_{1}}=G_{2}/U(2)$ is defined using the set of simple roots $\Theta=\Sigma-\{\alpha_{1}\}$ and has 2 isotropy summands $\got{m}_{1}$ and $\got{m}_{2}$ given by
$$
\begin{array}{l}
\got{m}_{1}=R^{+}(\alpha_{1},1)=\{\alpha_{1},\alpha_{1}+\alpha_{2},\alpha_{1}+2\alpha_{2},\alpha_{1}+3\alpha_{2}\},\\
\got{m}_{2}=R^{+}(\alpha_{1},2)=\{2\alpha_{1}+3\alpha_{2}\}.
\end{array}
$$

Therefore $\mathbb{F}_{G_{\alpha_{1}}}$ admits $2$ invariant almost complex structures, up to conjugation:
$$
\begin{array}{ccccccc}
J_{1}=(+,+)&&&\mbox{and}&&&J_{2}=(+,-).
\end{array}
$$

The cohomology ring of $\mathbb{F}_{G_{\alpha_{1}}}$ is given by 
$$
H^{*}(\mathbb{F}_{G_{\alpha_{1}}},\mathbb{R})=\dfrac{\mathbb{R}[x,y,z]}{\langle s_{1},s_{2},s_{3}\rangle},
$$
where $s_{1}=x+y+z$, $s_{2}=x^{2}+y^{2}+z^{2}$ and $s_{3}=x^{6}+y^{6}+z^{6}$.

The Chern numbers of the almost complex manifolds $(\mathbb{F}_{G_{\alpha_{1}}},J_{1})$ and $(\mathbb{F}_{G_{\alpha_{1}}},J_{2})$ are given in the table \ref{tabg22}.

\begin{table}[h] 

\centering

\caption{Chern numbers of $G_{2}/U(2)$ with 2 isotropy summands} 

\begin{tabular}{|c|c|c|} 

\hline 

&$(+,+)$ & $(+,-)$\\ 

\hline
\hline

$c_{5}$& 6 & -6 \\
\hline
$c_{1}^{5}$& 4374 & 9\\
\hline
$c_{1}^{3}c_{2}$ & 2106 & 3 \\ 
\hline
$c_{1}^{2}c_{3}$ & 594 & -9 \\
\hline
$c_{1}c_{4}$ & 90 & -9\\
\hline
$c_{1}c_{2}^{2}$ & 1014 & 1\\
\hline 
$c_{2}c_{3}$ & 286 & -9 \\
\hline
\end{tabular}
\label{tabg22}
\end{table}

According the Hirzebruch-Riemann-Roch Theorem the Chern numbers must satisfy the following relations:
$$
1440=-c_{1}c_{4}+c_{1}^{2}c_{3}+3c_{1}c_{2}^{2}-c_{1}^{3}c_{2}. \ \ 
$$

Therefore we have the following result

\begin{proposition}
The generalized flag manifold $G_{2}/U(2)$ with two isotropy summands has 2 invariant almost complex structures, up to conjugation and equivalence. One of these invariant almost complex structure is integrable.
\end{proposition}



\begin{thebibliography}{99}



\bibitem{naoformal} Amann M., Non-formal homogeneous spaces, {\sl Math. Z.}, 274 (3-4): 1299 1325, 2013.

\bibitem{AZ} Amann, M., Ziller, W., Geometrically formal homogeneous metrics of positive curvature. To appear J. Geom. Anal. (2015).

\bibitem{grego1}Arvanitoyeorgos, A., New invariant Einstein metrics on generalized flag manifolds,
{\sl Trans. Amer. Math. Soc.} 337(2) (1993) 981?995.

\bibitem{grego} Arvanitoyeorgos, A., Geometry of flag manifolds, {\sl International Journal of geometric methods in modern physics}, Vol. 3, Nos. \textbf{5} and \textbf{6} (2006) 957-974.

\bibitem{grego2} Arvanitoyeorgos, A., Chrysikos, I., and  Sakane, Y., Homogeneous Einstein metrics on $G_{2}/T$, {\sl Proceedings of the American Mathematical Society}, 141, \textbf{7}, (2013), 2485--2499.

\bibitem{besse} Besse, A., Einstein Manifolds, Springer Verlag.

\bibitem{opa} Borel A., Hirzebruch F., Characteristic classes and homogeneous spaces, I; II. {\sl Amer. J. Math.} \textbf{80} (1958), 458-538; \textbf{81} (1959), 315--382.

\bibitem{frances} Borel, A., Sur la cohomologie des espaces fibres principaux et des espaces homogenes de groupes de Lie compacts, {\sl Annals of Mathematics,} vol. 57 (1953), 115-207.

\bibitem{B-R}  Burstall,F.,  Rawnsley, J., {Twistor theory for Riemannian symmetric spaces.} Springer-Verlag, Berlin and New York (1990).




 






\bibitem{upper}  Kobayashi,  S., Nomizu, K.,
Foundations of Differential Geometry, Vol. II {\sl Wiley Interscience}, New York, 1969.


\bibitem{Kots} Kotschick, D., On products of harmonic forms, {\sl Duke Math. J. 107:3} (2001),
521-531. MR 2002c:53076 Zbl 1036.53030.


\bibitem{Kote} Kotschick, D., Terzi\'c, S., On formality of generalised symmetric spaces, {\sl Math. Proc. Cambridge Philos. Soc.} \textbf{134} (2003), 491-505. Zb1 1042.53035 MR 1981214.

\bibitem{princ} Kotschick, D., Terzi\'c, S., Chern numbers and geometry of partial flag manifolds, {\sl Comment. Math. Helv.} \textbf{84} (2009), 587-616.

\bibitem{Ko1} Kotschick, D., Terzi\'c, S., Geometric formality of homogeneous spaces and biquotients, {\sl Pacific J. Math.} \textbf{249} (2011), 157-176.


\bibitem{hum} Milnor, J. W., Stasheff, J. D., Characteristic classes, {\sl Princeton University Press} 1974.



\bibitem{parede} Gutierrez, M. P., Aspectos da geometria complexa das variedades flag, {\sl PhD. Thesis - University of Campinas (in Portuguese)}, 2000.

\bibitem{livro} San Martin, L. A. B., Algebras de Lie. {\sl Editora Unicamp}, 1999.


\bibitem{SM} San Martin, L. A. B., Negreiros, C. J. C., Invariant almost Hermitian structures on flag manifolds, {\sl  Advances in Mathematics} \textbf{178} (2003) 277-310.

\bibitem{rita} San Martin, L. A. B., Silva, R.C.J., Invariant nearly-K\"ahler structures, {\sl  Geom.Dedicata} \textbf{121} (2006) 143--154.





1966.

\bibitem{Sul} Sullivan, D., Differential forms and the topology of manifolds, {\sl pp. 37-49 in
Manifolds (Tokyo, 1973), edited by A. Hattori, Univ. Tokyo Press}, Tokyo, 1975. MR 51 6838
Zbl 0319.58005.

\bibitem{Tojo} Tojo, K., K\"ahler C-spaces and k-symmetric spaces, {\sl Osaka J. Math} \textbf{34} (1997) 803-820.





%
\end{thebibliography}
\end{document}